\documentclass[a4paper]{amsart}
\usepackage{amssymb,stmaryrd}
\usepackage{amscd}
\usepackage{slashed}

\usepackage{amsthm}
\theoremstyle{plain}
\newtheorem{prop}[]{Proposition}
\newtheorem{lem}[prop]{Lemma}
\newtheorem{thm}[]{Theorem}
\newtheorem{cor}[prop]{Corollary}
\theoremstyle{definition}
\newtheorem{dfn}{Definition}
\theoremstyle{remark}
\newtheorem{rmk}[]{Remark}
\newtheorem{example}[]{Example}

\newcommand{\ensemble}[1]{\left\{ #1 \right\}}
\newcommand{\suchthat}{\:|\:}
\newcommand{\norm}[1]{\left\| #1 \right\|}
\newcommand{\absolute}[1]{\left| #1 \right|}
\newcommand{\pairing}[2]{\langle #1, #2 \rangle}
\newcommand{\cptops}{\mathcal{K}}
\newcommand{\schwfuncs}{\mathcal{S}}
\newcommand{\N}{\mathbb{N}}
\newcommand{\Z}{\mathbb{Z}}
\newcommand{\C}{\mathbb{C}}
\newcommand{\R}{\mathbb{R}}
\newcommand{\T}{\mathbb{T}}
\newcommand{\fin}{\text{fin}}

\newcommand{\vol}{\text{vol}}
\newcommand{\smooth}[1]{\mathcal{#1}}

\DeclareMathOperator{\Tr}{Tr}
\DeclareMathOperator{\Ad}{Ad}
\DeclareMathOperator{\End}{End}
\DeclareMathOperator{\dom}{dom}
\DeclareMathOperator{\ch}{ch}
\DeclareMathOperator{\rk}{rk}
\DeclareMathOperator{\ev}{ev}
\DeclareMathOperator{\Id}{Id}
\DeclareMathOperator{\ptensor}{\hat{\otimes}}
\DeclareMathOperator{\mintensor}{\otimes_{\min}}
\DeclareMathOperator{\htensor}{\bar{\otimes}}
\DeclareMathOperator{\HP}{HP}
\DeclareMathOperator{\HC}{HC}
\DeclareMathOperator{\KK}{KK}

\title{Connes-Landi Deformation of Spectral Triples}
\author{Makoto Yamashita}
\address{Graduate School of Mathematical Sciences, the University of Tokyo\\
3-8-1 Komaba Meguro-ku, Tokyo, JAPAN}
\email{makotoy@ms.u-tokyo.ac.jp}
\begin{document}
\begin{abstract}
 We describe a way to deform spectral triples with a 2-torus action and a real deformation parameter, motivated by deformation of manifolds after Connes-Landi.  Such deformations are shown to have naturally isomorphic $K$-theoretic invariants independent of the deformation parameter.
\end{abstract}

\maketitle
\section{Introduction}
%\subsection{}
Let $M$ be a compact smooth Riemannian manifold endowed with a smooth action of $2$-torus $\T^2 = (\R/\Z)^2$.  Connes and Landi \cite{MR1846904} defined an isospectral deformation $M_\theta$ of $M$ for a deformation parameter $\theta \in \R/\Z$.  The ``smooth function algebra'' $C^\infty(M_\theta)$ of $M_\theta$ is, as a linear space given by the smooth function algebra $C^\infty(M)$ of $M$, but endowed with a deformed product
\[
f *_\theta g = e^{\pi i \theta (m n' - m' n)} f g
\]
when $f$ is an $\T^2$-eigenvector of weight $(m, n) \in \Z^2$ in $C^\infty(M)$ and $g$ is a one of weight $(m', n')$.  In the case where $M = \T^2$ and the action is the translation, one obtains the noncommutative torus $\T^2_\theta$ whose function algebra is generated by the two unitaries $u$ and $v$ subject to the relation $u v = e^{2\pi i \theta}v u$.  They also showed that Dirac type elliptic operator over $M$ determined by the metric structure on $M$ such as the signature operator or the spin Dirac operator continue to make sense over $M_\theta$ as unbounded self adjoint operators of compact resolvent over Hilbert spaces on which $C^\infty(M_\theta)$ acts.  Li \cite{MR2134342} showed that the spectral triple defined this way qualifies as a ``compact quantum metric space'' of Rieffel \cite{MR1647515}.

This construction can be easily generalized to a spectral triple $(\smooth{A}, H, D)$ over a noncommutative algebra endowed with a smooth action of $\T^2$ when the ``Dirac operotor'' $D$ is equivariant with respect to this action.  In such a case obtain a new spectral triple  $(\smooth{A}_\theta, H, D)$  over the deformed algebra exactly as in the smooth Riemannian manifold case.  The $C^*$-algebraic closure of $\smooth{A}_\theta$ is a particular case of the deformation considered in \cite{MR1237992}, where such a deformation is shown to have the same $K$-group as the original algebra.

One question which arises after such deformation is to compute the pairing of the Chern-Connes character of thus obtained spectral triple with the $K$-group of the deformed algebra.  In the case of noncommutative torus $\T^2_\theta$, Pimsner and Voiculescu showed that it has the same $K$-theory as the ordinary $2$-torus in \cite{MR587369}, then Connes \cite{MR572645} \cite{MR823176} made the computations of the periodic cyclic cohomology group and the pairing with the $K_0$-group in terms of the connections of projective modules as classified by Rieffel \cite{MR941652}.

In this paper we show that 1) Connes-Landi deformation gives isomorphic $\HP$-group as the original algebra, compatible with the $K$-theory isomorphism (Corollary \ref{prop:per-cyclic-isom}), 2) deformation of a spectral triple give the pairing of the same image (Theorem \ref{thm:cl-deform-pairing-invar}), somewhat generalizing the above computations for $\T^2_\theta$ to the general $M_\theta$.  In the course of the proof of the latter invariance of Chern-Connes characters we describe the image of the invariant cyclic cocycles under the former isomorphism.  In doing so we obtain a simple description of the phenomenon like $\pairing{\tau^{(\theta)}}{K_0 C(\T^2_\theta)} = \Z + \theta \Z$ for the gauge invariant tracial state $\tau^{(\theta)}$ on $C(\T^2_\theta)$ which is sensitive to the change of deformation parameter as opposed to Chern-Connes character of equivariant spectral triples.

\section{Preliminaries}

In this section we give a basic definition of Connes-Landi deformation of spectral triples with a $2$-torus action and related constructions.  Throughout this paper we consider regular spectral triple (\cite{MR1789831}, \cite{MR1995874}) with an additional assumption on the smoothness of the torus action.

Let $(\smooth{A}, H, D)$ be an even spectral triple (in other words, a $K$-cycle over $\smooth{A}$); thus, $H = H^0 \oplus H^1$ is a graded Hilbert space, $\smooth{A}$ is a $*$-subalgebra of the algebra of the bounded even operators $B(H^0) \oplus B(H^1)$ on $H$ and $D$ is an odd unbounded self adjoint on $H$ which has bounded commutator with the elements $\smooth{A}$: $[D, a]$ is bounded and $a(1+D^2)^{-1/2}$ is compact for any $a \in \smooth{A}$.

Let $\delta$ denote the densely defined closed derivation $T \mapsto [\absolute{D}, T]$ on $B(H)$.  Recall that $(\smooth{A}, H, D)$ is said to be regular when $\smooth{A} + [D, \smooth{A}] \subset \cap_{k = 1}^\infty \dom \delta^k$.  Let $A$ denote the operator norm closure of $\smooth{A}$ in $B(H)$.  The completion $\smooth{A}_\delta$ of $\smooth{A}$, with respect to the seminorms $\norm{\delta^k(-)}$ and $\norm{\delta^k([D, -])}$ for $k \in \N = \ensemble{0, 1, \ldots}$, is stable under holomorphic functional calculus inside $A$ (\cite{MR1995874}, Proposition 16).
Let $\sigma$ be an action of $\T^2$ on $\smooth{A}$ by $*$-automorphisms.  In the following we assume that it is strongly continuous with respect to the Fr\'echet topology on $\smooth{A}$, i.e. for any $a \in \smooth{A}$ the map $t \mapsto \sigma_t(a)$ is continuous with respect to the seminorms on $\smooth{A}$ mentioned above.  Now, suppose that $\sigma_t$ is spatially implemented on $H$ by a strongly continuous even unitary representation $U_t\colon \T^2 \rightarrow \mathcal{U}(H^0) \times \mathcal{U}(H^1)$ on $H$ satisfying
\begin{align}\label{eq:cov-rep-cl}
\sigma_t(a) &= \Ad_{U_t}(a), &\Ad_{U_t}(D) &= D
\end{align}
 for $t \in \T^2$ and $a \in \smooth{A}$.  This condition implies that $\sigma$ is isometric with respect to the seminorms on $\smooth{A}$ mentioned above:
\begin{align}\label{eq:unif-bound-action}
\norm{\delta^k{\sigma_t(a)}} &= \norm{\delta^k(a)},& \norm{\delta^k([D, \sigma_t(a)])} &= \norm{\delta^k([D, a])}.
\end{align}

Put $U^{(1)}_t = U_{(t, 0)}$ and $U^{(2)}_t = U_{(0, t)}$.  For $i = 1, 2$, let $h_i$ denote the generator
\[
h_i \xi = \lim_{t \rightarrow 0} \frac{U^{(1)}_t \xi - \xi}{t}
\]
and put $\sigma^{(i)}_t = \Ad_{U^{(i)}_t} |_{\smooth{A}}$.  In addition, we write $[h, a]^{(\alpha)}$ for the iterated derivation
\[
\underbrace{[h_1, \cdots , [h_1}_{\alpha_1 \times}, \underbrace{[h_2, \cdots, [h_2}_{\alpha_2 \times}, a] \cdots ] ]\cdots ].
\]
Since the derivations $[h_i, -]$ are closable and have dense domains on $A$, $\sigma$ extends to a strongly continuous action of $\T^2$ on $A$ by Theorems 1.4.9 and 1.5.4 of  \cite{MR871870}.

Let $\smooth{A}^{\infty}$ be the subalgebra of $\smooth{A}_\delta$ consisting of the elements $a$ such that the map $t \mapsto \sigma_t(a)$ admit arbitrary order of derivatives in $\smooth{A}_\delta$.  The algebra $\smooth{A}^\infty$ is stable under holomorphic functional calculus inside $A$.  On the other hand, by the strong continuity of $\sigma$ and (\ref{eq:unif-bound-action}), the element
\[
\sigma_f(a) = \int_{\T^2} f(t) \sigma_t(a) d t
\]
is contained in $\smooth{A}_\delta$ for any smooth function $f$ on $\T^2$ and any $a \in \smooth{A}_\delta$.  Thus $\smooth{A}^\infty$ is a dense subalgebra of $A$ which contains $\smooth{A}$ and is closed under the holomorphic functional calculus.  Moreover it can be characterized as the joint domain of the following densely defined seminorms on $A$;
\begin{equation}\label{eq:smooth-seminorms}
\nu_{k, \alpha}(a) = \norm{\delta^k([h, a]^{(\alpha)})} + \norm{\delta^k([D, [h, a]^{(\alpha)}])}
\end{equation}
for $k \in \N$ and $\alpha = (\alpha_1, \alpha_2) \in \N^2$.

\begin{rmk}
 When we define the algebra $\smooth{A}^\infty$, the higher derivatives of the functions $t \mapsto \sigma_t(a)$ for $a \in \smooth{A}^\infty$ are only required to have their derivatives of arbitrary order in $\smooth{A}_\delta$ with respect to the operator norm topology.  Then for any $a \smooth{A}^\infty$ there exist derivatives of $\sigma_t(a)$ with respect to the seminorms $\norm{\delta^k(a)} + \norm{\delta^k([D, a])}$ for $k \in \N$.  Indeed, for $i = 1, 2$, we know that $\partial_1 \sigma_t(a)$ exists with respect to the operator norm, equals $[h_i, \sigma_t(a)]$, and it is an element of $\smooth{A}_\delta$.  For any $k$ the operator $\delta^k([h_i,  \sigma_t(a)])$ is bounded and equals $[h_i, \sigma_t(\delta^k(a))]$.  Now, one has
 \[
\delta^k( \sigma_t(a)) = \delta^k(a) + \int_0^{t_1} \delta^k([h_1, \sigma^{(1)}_r(a)]) d r + \int_0^{t_2} \delta^k([h_2, \sigma^{(2)}_s \sigma^{(1)}_{t_1}(a)]) d s,
\]
 which shows $\partial_2 \delta^k( \sigma_t(a)) = \delta^k([h_2, \sigma_t(a)])$.  This shows that the map $\sigma_t(a)$ is differentiable by $\partial_2$ with respet to the seminorm $\norm{\delta^k(a)}$ and that the partial derivative is $[h_2, \sigma_t(a)]$.  With an analogous argument one has $\partial_1 \sigma_t(a) = [h_1, \sigma_t(a)]$ with respect to this seminorm.  Next, by induction on the order $\absolute{\alpha}$ of differentiation, one obtains that the higher order derivatives of $\sigma_t(a)$ exists with respect to the seminorm $\norm{\delta^k(a)}$ and agrees with the ones with respect to the operator norm case.  The case for the seminorms $\norm{\delta^k([D, a])}$ fro $k \in \N$ is similar.
\end{rmk}

By the stability under the holomorphic functional calculus, the change of algebras from $\smooth{A}_\delta$ to $\smooth{A}^\infty$ does not affect the $K_0$-group (which is isomorphic to $K_0(A)$) and they have the same $K$-cycle given by $H$ and $D$.   In the rest of the paper we assume that $\smooth{A} = \smooth{A}^\infty$.  

 The Hilbert spaces $H^i$ for $i = 0, 1$ decomposes into a direct sum of eigenspaces $H^i_{(m, n)}$ of weights $\hat{\T}^2 \simeq \Z^2$ characterized by
 \[
 \xi \in H^i_{(m, n)} \Leftrightarrow U_t \xi = e^{2 \pi i (m t_1 + n t_2)}\xi
 \]
for $t = (t_1, t_2) \in \T^2$.  Similarly, for any bounded operator $T$ on $H$, put
\begin{equation}\label{eq:t-mn-component-integral}
T_{(m, n)} = \int_{\T^2} e^{- 2 \pi i (m t_1 + n t_2)} \Ad_{U_t}(T) d\mu(t)
\end{equation}
where $\mu$ is the normalized Haar measure on $\T^2$.  Then $T_{(m, n)}$ satisfies $\Ad_{U_t}(T_{(m, n)}) = e^{2\pi i (m t_1 + n t_2)}T_{(m, n)}$ and $T$ can be expressed as a sum $\sum_{(m, n)} T_{(m, n)}$ which converges in the strong operator topology. 
 
Let $B(H)_\fin$ denote the subspace of $B(H)$ consisting of the operators $T$ where $T_{(m, n)} = 0$ except for finitely many $(m, n)$.  Put $\smooth{A}_\fin = \smooth{A} \cap B(H)_\fin$.

\begin{dfn}
  Let $\theta$ be an arbitrary real number.   Given a bounded operator $T$ on $H$ which is an eigenvector of weight $(m, n) \in \Z^2$ for the action $\Ad_{U_t}$ of $\T^2$, we define a new bounded operator $T^{(\theta)}$ on $H$ by
\[
 T^{(\theta)} \xi = e^{\pi i \theta (m n' - m' n)} T \xi
\]
for $\xi \in H_{(m', n')}$.  We extend this to the operators in $B(H)_\fin$ by putting $T^{(\theta)} = \sum T_{(m, n)}^{(\theta)}$.
\end{dfn}

Note that we have $(T^{(\theta)})^* = (T^*)^{(\theta)}$.  When $T$ is homogeneous of weight $(m, n)$ and $S$ is a one of weight $(m', n')$, we have $T^{(\theta)} S^{(\theta)} = e^{\pi i (m n' - m' n) \theta} (TS)^{(\theta)}$.  

\begin{rmk}
 We adopted a presentation of $T^{(\theta)}$ which is slightly different from the one given in \cite{MR1846904}.  Let $V$ be the unitary operator on $H$ characterized by $V \xi = e^{\pi i \theta m' n'} \xi$ for $\xi \in H_{(m', n')}$, and $\phi$ be the linear transformation of $B(H)_\fin$ characterized by $\phi(T) = e^{- \pi i \theta m n} T$ when $T \in B(H)_{(m, n)}$.  Then we have
 \[
 V \phi(T)^{(\theta)} V^* \xi = e^{2 \pi i \theta m n'} T \xi,
 \]
 which agrees with the deformation given by Connes and Landi.
\end{rmk}

\begin{lem}\label{lem:part-sum-conv}
Let $T$ be a bounded operator on $H$ which admits derivatives of the map $t \mapsto \Ad_{U_t}(T)$ up to the fourth degree in $B(H)$ in the weak operator topology.  Then the sum $\sum_{(m, n) \in \Z^2} T_{(m, n)}^{(\theta)}$ is absolutely convergent to a bounded operator $T^{(\theta)}$ in the operator norm topology.
\end{lem}

\begin{proof}
For $i = 1, 2$, let $\partial_i$ denote the partial differentiation in the direction of $t_i$ of functions $F(t_1, t_2)$ defined on $\T^2$.  Let $(m, n)$ be any element of $\Z^2$. We claim that the bounded operator $S = \left . \partial_1 \Ad_{U_t}(T)\right |_{t=(0,0)}$ satisfies $S_{m, n} = m T_{m, n}$.  Indeed, when $\xi \in H_{m', n'}$ and $\eta \in H_{m'', n''}$, we have
\[
\frac{\pairing{\Ad_{U^{(1)}_s}(T)\xi}{\eta} - \pairing{T \xi}{\eta}}{s} = \frac{e^{2 \pi i (m'' - m') s} - 1}{s} \pairing{T \xi}{\eta}.
\]
The left hand side converges to $\pairing{S \xi}{\eta}$ as $s \rightarrow 0$, while the right hand side converges to $2 \pi i (m'' - m') \pairing{T \xi}{\eta}$.  Hence we have that
\[
e^{-2 \pi i (m t_1 + n t_2)} \pairing{\Ad{U_{t}}(S) \xi}{\eta} = 2 \pi i (m'' - m') e^{-2 \pi i (\tilde{m} t_1 + \tilde{n} t_2)}  \pairing{T \xi}{\eta},
\]
where one puts $\tilde{m} = -m - m' + m''$ and $\tilde{n} = -n - n' + n''$.  This leads to 
\[
\pairing{S_{m, n} \xi}{\eta} =
\begin{cases}
2 \pi i (m'' - m') \pairing{T \xi}{\eta} & (m = m'' - m', n = n'' - n')\\
0 & (\text{otherwise}).
\end{cases}
\]
This agrees with the value of $2 \pi i m \pairing{T_{m, n} \xi}{\eta}$.  Hence the bounded operators $S_{m, n}$ and $2 \pi i m T_{m, n}$ agree on the linear spans of the $H_{m', n'}$ for $(m', n')$ in $H$.  Since this subspace is dense, we have established the claim.

Iterating the above argument, we obtain 
\[
(\left . (\partial_1^2 + \partial_2^2)^2 \Ad_{U_t}(T)\right |_{t=(0,0)})_{m, n} = 16 \pi^4 (m^2 + n^2)^2 T_{(m, n)}
\]
for any $(m, n) \in \Z^2$.  Since the correspondence $T \mapsto T_{(m, n)}$ is a contraction, we have
\[
\norm{T_{(m, n)}} \le \frac{1}{16 \pi^4(m^2 + n^2)}\norm{\left . (\partial_1^2 + \partial_2^2)^2 \Ad_{U_t}(T)\right |_{t=(0,0)}}.
\]
 for $(m, n) \in \Z^2 \setminus \ensemble{(0, 0)}$ Now, the assertion of Lemma follows from the fact that $(m^2 + n^2)^{-2}$ is summable on $\Z^2$ and $\norm{T_{(m, n)}^{(\theta)}} = \norm{T_{(m, n)}}$.
\end{proof}

\begin{lem}
 The subspace $\smooth{A}_\theta \subset B(H)$ of the operators $a^{(\theta)}$ for $a \in \smooth{A}$ is closed under multiplication.
 \end{lem}

\begin{proof}
 Let $a$ and $b$ be arbitrary elements of $\smooth{A}$.  We than have
 \begin{equation}\label{eq:prod-deform-series}
  a^{(\theta)} b^{(\theta)} = \left ( \sum_{(m, n), (m', n') \in \Z^2} e^{\pi i \theta (m n' - m' n)} a_{m, n} b_{m', n'} \right )^{(\theta)}.
 \end{equation}
In order to show that the series in the the right hand side defines an element of $\smooth{A} = \smooth{A}^\infty$, it is enough to show that it is uniformly convergent with respect to any of the seminorms $\nu_{k, \alpha}$ for $k \in \N$ and $\alpha \in \N^2$.

By induction on $k$ and $\absolute{\alpha}$, we know that there exist numbers $C^{k, \alpha}_{l, \beta}$ indexed by $k, l \in \N$ and $\alpha, \beta \in \N^2$ satisfying
\[
\delta^k([h, a'b']^{(\alpha)}) = \sum_{l \le k, \beta \le \alpha} C^{k, \alpha}_{l, \beta} \delta^l([h, a']^{(\beta)}) \delta^{k - l}([h, b']^{(\alpha - \beta)})
\]
for any elements $a', b' \in \smooth{A}$.  Let us fix $k \in \N$ and $\alpha \in \N^2$ now.  Then, for any $l \le k$ and $\beta \le \alpha$, the functions $\sigma_t(\delta^l([h, a]^{(\beta)}))$ and $\sigma_t(\delta^l([h, b]^{(\beta)}))$ admit bounded derivatives in $t \in \T^2$ of order $4$.  Hence we have the absolute convergence of
\[
\sum_{(m, n) \in \Z^2} \delta^l([h, a_{m, n}]^{(\beta)})
\]
by Lemma \ref{lem:part-sum-conv}.  Hence the infinite series
\[
\sum_{(m, n), (m', n') \in \Z^2} \delta^l([h, a_{m, n}]^{(\beta)}) \delta^{k - l}([h, b_{m', n'}]^{(\alpha - \beta)})
\]
is also absolutely convergent, which implies the convergence of
\[
\sum_{(m, n), (m', n') \in \Z^2} e^{\pi i \theta (m n' - m' n)} \delta^l([h, a_{m, n}]^{(\beta)}) \delta^{k - l}([h, b_{m', n'}]^{(\alpha - \beta)}).
\]
Combining this for all the indices $l \le k$ and $\beta \le \alpha$, one obtains the convergence of the right hand side of \eqref{eq:prod-deform-series} with respect to the seminorm $\delta^k([h, a']^{(\alpha)})$ for $a' \in \smooth{A}$.  With a similar argument we also have the convergence of that series with respect to the seminorm $\delta^k([D, [h, a']^{(\alpha)}])$, hence also with respect to $\nu_{k, \alpha}$. This proves the assertion of Lemma.
\end{proof}

\begin{dfn}\label{dfn:cl-deformations}
 Let $\smooth{A}, H, D, \sigma$ and $U$ be as above and $\theta$ be an arbitrary real number.  The algebra of the operators $a^{(\theta)}$ for $a \in \smooth{A}$ is called the \textit{Connes-Landi deformation} $\smooth{A}_\theta$ of $\smooth{A}$.  The operator norm closure of $\smooth{A}_\theta$ inside $B(H)$ is called the Connes-Landi deformation $A_\theta$ of $A$. \end{dfn}

\begin{rmk}
 Let $\sigma$ be an action of $\R^d$ on a $C^*$-algebra $A$ and $J$  a skew symmetric matrix of size $d$.  Rieffel \cite{MR1184061} defined a deformed product
\begin{equation}\label{eq:rieffel-deform-prod}
 a \times_J b = \int \sigma_{J u}(a) \sigma_v(b) d u d v 
\end{equation}
on the $\sigma$-smooth part $A^\infty$ of $A$ by means of oscillatory integral.  In our setting, the action $\sigma$ of the $2$-torus induces an action of $\R^2$. For the $2 \times 2$ skew symmetric matrix $J$, consider the following matrix (cf. loc. cit. Example 10.2)
 \begin{equation*}
 J = \left(\begin{array}{cc}0 & - \frac{\theta}{2} \\ \frac{\theta}{2} &0 \end{array}\right).
 \end{equation*}
Then the deformed $C^*$-algebra $(\overline{A^\infty}, \times_J)$, which is obtained as the $C^*$-algebraic closure of $(A^\infty, \times_J)$, is isomorphic to $A_\theta$.  Hence the presentation of the latter on $H$ as in Definition \ref{dfn:cl-deformations} gives a representation of $(\overline{A^\infty}, \times_J)$ on the graded Hilbert space $H$ as even operators.
 \end{rmk}

Note that the equation (\ref{eq:rieffel-deform-prod}) defines a product on $\smooth{A}$ under our assumption.  The correspondence $a \mapsto a^{(\theta)}$ gives a representation of $(\smooth{A}, \times_J)$ on $H$.  In the remaining of the section we show that $\smooth{A}_\theta$ satisfies the conditions assumed for $\smooth{A}$.

\begin{lem}\label{lem:fin-supp-regular}
 For any $a \in \smooth{A}$ and any $\theta \in \R$, the operators $\delta^k([h, a^{(\theta)}]^{(\alpha)})$ and $\delta^k([D, [h, a^{(\theta)}]^{(\alpha)}])$ are bounded for $k \in \N$ and any $\alpha \in \N^2$.
\end{lem}

\begin{proof}
  Applying Lemma \ref{lem:part-sum-conv} to $T = \delta^k([h, a]^{(\alpha)})$, one has the absolute convergence of $\sum_{m, n} \delta^k([h, a]^{(\alpha)})_{(m, n)}^{(\theta)}$.  Since $\delta^k([h, a]^{(\alpha)})_{(m, n)}^{(\theta)} = \delta^k([h, a_{(m, n)}^{(\theta)}]^{(\alpha)})$ and $T \mapsto \delta^k([h, T]^{(\alpha)})$ is closed on $B(H)$, one has $a^{(\theta)} \in \dom \delta^k([h, -]^{(\alpha)})$ and $\delta^k([h, a^{(\theta)}]^{(\alpha)}) = \delta^k([h, a]^{(\alpha)})^{(\theta)}$.
 \end{proof}

\begin{lem}\label{lem:a-theta-d-inv-cpt}
 For any $a \in \smooth{A}$ and $\theta \in \R$, the operator $a^{(\theta)} (1 + D^2)^{-1/2}$ on $H$ is compact.
\end{lem}

\begin{proof}
The operator $a_{(m, n)} (1 + D^2)^{-1/2}$, and $a_{(m, n)}^{(\theta)}(1 + D^2)^{-1/2}$ which is similar to the former, are compact for any pair $(m, n) \in \Z^2$. Applying Lemma \ref{lem:part-sum-conv} to the operator $a (1 + D^2)^{-1/2}$, we know that
\[
a^{(\theta)} (1 + D^2)^{-1/2} = \lim_{N\rightarrow\infty} \sum_{\absolute{m}, \absolute{n} < N} a_{m, n}^{(\theta)}(1 + D^2)^{-1/2}
\]
is also compact.
\end{proof}

\begin{prop}
The triple $(\smooth{A}_\theta, H, D)$ is an even regular spectral triple.  The action $\Ad_{U_t}$ on $\smooth{A}_\theta$ by $\T^2$ is smooth and $\smooth{A}_\theta$ is complete with respect to the seminorms $\nu_{k, \alpha}$ for $k \in \N$ and $\alpha \in \N$.
\end{prop}

\begin{proof}
The fact that $(\smooth{A}_\theta, H, D)$ is a regular spectral and the smoothness of $\Ad_{U_t}$ on $\smooth{A}_\theta$ follows from Lemmas \ref{lem:fin-supp-regular} and \ref{lem:a-theta-d-inv-cpt}.

 Let $(a_k^{(\theta)})_{k \in \N}$ be a Cauchy sequence with respect to the seminorms $\nu_{k, \alpha}$ in $\smooth{A}_\theta$, convergent to a bounded operator $S$ on $H$.   Then the map $t \mapsto \Ad_{U_t}(S)$ is smooth on $\T^2$.  Hence we have a bounded operator $S^{(-\theta)}$ on $H$.  It remains to show that $S^{(-\theta)} \in \smooth{A}_\delta$, since one would have $S = (S^{(-\theta)})^{(\theta)} \in \smooth{A}_\theta$ then.
 
 Let $m \in \N$ and $\alpha \in \N^2$.  As in the proof of Lemma \ref{lem:part-sum-conv}, there is a universal constant $C$ such that
 \[
 \norm{\delta^m([h, (a_k^{(\theta)} - a_{k'}^{(\theta)})^{(-\theta)}]^{(\alpha)})} \le C \norm{\Ad_{U_t}(\delta^m([h, a_k^{(\theta)} - a_{k'}^{(\theta)})]^{(\alpha)})}_{C^4(\T^2)}
 \]
  for any $k, k'$.  Thus we obtain
  \[
  \norm{\delta^m([h, a_k - a_{k'}]^{(\alpha)})} < C \max_{\absolute{\alpha'} \le 4} \nu_{m, \alpha + \alpha'}(a_k^{(\theta)} - a_{k'}^{(\theta)}).
  \]
  There is also a similar estimate for $\norm{\delta^m([D, a_k - a_{k'}])}$.  Hence the sequence $(a_k)_{k}$ in $\smooth{A}$ is a Cauchy sequence for the seminorms $\nu_{k, \alpha}$, which is convergent to $S^{(-\theta)}$.  This shows $S^{(-\theta)} \in \smooth{A}_\delta$.
\end{proof}

\begin{rmk}
 The newly obtained spectral triple $(\smooth{A}_\theta, H, D)$ again satisfies $\smooth{A}_\theta = \smooth{A}_\theta^\infty$.  Hence one can form $(\smooth{A}_\theta)_{\theta'}$ for yet another deformation parameter $\theta'$, which is identified to $\smooth{A}_{\theta + \theta'}$.  We also have $\smooth{A}_0 = \smooth{A}$.
\end{rmk}

\begin{example}
 Let $\smooth{A}$ be the smooth function algebra $C^\infty(\T^2)$ over the $2$-torus, $H = L^2(\T^2)^{\oplus 2}$ and
 \[
 \slashed{D} = \left[\begin{array}{cc}0 & i \partial_1 + \partial_2 \\i \partial_1 - \partial_2 & 0\end{array}\right].
 \]
Consider the action of $\T^2$ on $\smooth{A}$ given by the translation.  Then the spectral triple $(\smooth{A}, H, \slashed{D})$ is regular and satisfies $\smooth{A} = \smooth{A}^\infty$. Given a real parameter $\theta$, the corresponding regular Connes-Landi deformation $\smooth{A}_\theta$ is precisely the algebra of Laurent series $\sum_{(m, n) \in \Z^2} a_{(m, n)} u^m v^n$ of rapid decay coefficients over the two unitaries $u$ and $v$ satisfying $u v = e^{2\pi i \theta}v u$.
\end{example}

\section{$K$-theory of Connes-Landi deformation}

We keep the notations $\smooth{A}, A, \sigma, \smooth{A}_\theta$ and $A_\theta$ of the previous section.  We extend $\sigma$ to $A$.  As remarked in the previous section, as an abstract $C^*$-algebra $A_\theta$ is the deformation algebra studied by Rieffel in \cite{MR1184061}, \cite{MR1237992} and shown to have isomorphic $K$-group as $A$.  In this section we elaborate a similar crossed product construction in order to show that $\smooth{A}_\theta$ and $\smooth{A}$ have isomorphic periodic cyclic cohomology groups which is compatible with the isomorphism between the $K$-groups.

Note that there is an action $\gamma$ of $\T^2$ on $C(\T^2_\theta)$ called the gauge action, given by
\[
\gamma_t (u^a v^b) = e^{2\pi i (t_1 a + t_2 b)} u^a v^b.
\]
We obtain the diagonal product action $\sigma \otimes \gamma$ of $\T^2$ on the (minimal) tensor product $A \mintensor C(\T^2_\theta)$.

\begin{prop}\label{prop:cl-deform-isom-to-fixed-pt-diag}
 The algebra $A_\theta$ is isomorphic to the subalgebra of $A \mintensor C(\T^2_\theta)$ consisting of the fixed elements under the diagonal action $\sigma \otimes \gamma$ of $\T^2$.
\end{prop}

\begin{proof}
 When $a \in A$ is a $\T^2$-eigenvector of weight $(m, n)$, $a \otimes u^{-m}v^{-n}$ is in the fixed point algebra $(A \mintensor C(\T^2_\theta))^{\sigma \otimes \gamma}$.  The correspondence $a^{(\theta)} \mapsto e^{- \pi i m n \theta} a \otimes u^{-m}v^{-n}$ defines a multiplicative map $f$ from $\smooth{A}_\theta \cap B(H)_\fin$ into $(A \mintensor C(\T^2_\theta))^{\sigma \otimes \gamma}$.
 
 Let $a$ be an element of $\smooth{A}$.  By the estimate given by Lemma \ref{lem:part-sum-conv}, one has the absolute convergence of $\sum_{m, n} a_{m,n} u^{-m}v^{-n}$ inside $(A \mintensor C(\T^2_\theta))^{\sigma \otimes \gamma}$.  Hence $f$ extends to the subalgebra $\smooth{A}_\theta$ of $A_\theta$.  Since $\smooth{A}_\theta$ is stable under holomorphic functional calculus, the spectral radius of $a^{(\theta)}$ in $\smooth{A}$ is the same as that in $A_\theta$.  It means that $f$ is a contraction and extends to a $*$-homomorphism of $A_\theta$ into $(A \mintensor C(\T^2_\theta))^{\sigma \otimes \gamma}$.
 
 We first prove that $f$ is injective by contradiction. Suppose that there was a nonzero positive element  $x$ in the ideal $\ker f$ of $A_\theta$.  On one hand we have $f(x_{0, 0}) = x_{0, 0} \otimes 1 \neq 0$.  On the other hand, since $f$ is equivariant for the action $\Ad_{U_t}$ on $A_\theta$ and $\sigma \otimes 1$ on $(A \mintensor C(\T^2_\theta))^{\sigma \otimes \gamma}$.  Hence the ideal $\ker f$ is invariant under $\Ad_{U_t}$ and $x_{(0, 0)} \in \ker f$.  This is a contradiction, hence $f$ is injective.
 
    Next we prove the surjectivity of $f$. By construction its image contains the linear span $E$ of the elements of the form $a \otimes u^{-m}v^{-n}$ for where $(m, n) \in \Z^2$ and $a \in \smooth{A}$ is a $\T^2$-eigenvector of weight $(m, n)$.  Since any homomorphism between $C^*$-algebras has a closed image, it is enough to show that the closure of $E$ agrees with $(A \mintensor C(\T^2_\theta))^{\sigma \otimes \gamma}$.
    
   First, note that if a sequence $(a_k)_{k \in \N}$ in $\smooth{A}$ converges to a $\T^2$-eigenvector $a$ in $A$ of weight $(m, n) \in \Z^2$, so does the sequence $((a_k)_{m, n})_{k \in \N}$.  Hence for any $\T^2$-eigenvector $a \in A$ of weight $(m, n)$, the element $a \otimes u^{-m}v^{-n}$ lies in the closure of $E$.
   
   Next, given any element $x$ in $(A \mintensor C(\T^2_\theta))^{\sigma \otimes \gamma}$, eigenspace decomposition with respect to the action $\sigma \otimes 1$ gives us the $\T^2$-eigenvectors $(a_{m, n})_{(m, n) \in \Z^2}$ satisfying
   \[
   a_{m, n} \otimes u^{-m}v^{-n} = e^{2 \pi i (-m t_1 - n t_2)} \int_{\T^2} (\sigma_t \otimes 1)(x).
   \]
   Let $c^{(k)}_{m, n}$ be the sequence finitely supported coefficients defined by
   \[
    c^{(k)}_{m, n} = \frac{\absolute{ ([ m, k + m] \times [ n, k + n] )\cap( [0, k] \times [0, k] )\cap \Z^2}}{k^2}.
   \]
  By the Fej\'{e}r kernel argument as in \cite{MR2511867}, the sequence
   \[
   \sum_{(m, n) \in \Z^2} c^{(k)}_{m, n} a_{m, n} \otimes u^{-m}v^{-n},
   \]
   converges to $x$ in norm. This shows that $E$ is dense in $(A \mintensor C(\T^2_\theta))^{\sigma \otimes \gamma}$.
\end{proof}

For $i = 1, 2$, let $\widehat{\sigma^{(i)}}$ denote the automorphism of $\T^2 \ltimes_{\sigma} A$ which is dual to $\sigma^{(i)}$.  Let $L^1(\T^2, A; \sigma)$ denote the algebra of $A$-valued measurable integrable functions on $\T^2$ endowed with the $\sigma$-convolution product $(f * g)_t = \int f_s \sigma_{s}(g_{t-s}) d s$.  Then the crossed product $\Z \ltimes_{\hat{\sigma^{(2)}}} \T^2 \ltimes_{\sigma} A$ by $\hat{\sigma^{(2)}}$ can be described as the $C^*$-algebraic closure of the normed $*$-algebra of the Laurent polynomials $\sum_k f_k v^k$ with coefficients in $L^1(\T^2, A; \sigma)$, whose product is given by
\[
\sum f_m v^m * \sum g_n v^n = \sum f_m * \hat{\sigma}_2^m(g_n) v^{m+n}.
\]

The algebra $\Z \ltimes_{\hat{\sigma^{(2)}}} \T^2 \ltimes_{\sigma} A$ admits two automorphisms of interest, the first being
\[
\widehat{\sigma^{(1)}}(\sum f_k v^k) = \sum \widehat{\sigma^{(1)}}(f_k) v^k
\]
 and the second being
 \[
 \alpha_\theta(\sum f_k v^k) = \sum e^{2 \pi i k \theta} f_k v^k.
 \]
 Let $(\widehat{\sigma^{(1)}}, \alpha_\theta)$ denote an automorphism of $\Z \ltimes_{\hat{\sigma^{(2)}}} \T^2 \ltimes_{\sigma} A$ given by the composition of these two automorphisms.

\begin{lem}
 The algebra $A_\theta$ is strongly Morita equivalent to $\Z \ltimes_{\widehat{\sigma^{(1)}}, \alpha_\theta} \Z \ltimes_{\hat{\sigma^{(2)}}} \T^2 \ltimes_{\sigma} A$.
\end{lem}

\begin{proof}
 Since the action of $\T^2$ on $C(\T^2_\theta)$ has full spectrum, so does the one on $A \otimes C(\T^2_\theta)$.  Hence  $A_\theta$, which is isomorphic to $(A \mintensor C(\T^2_\theta))^{\T^2}$ by Proposition \ref{prop:cl-deform-isom-to-fixed-pt-diag}, is strongly Morita equivalent to $\T^2 \ltimes (A \mintensor C(\T^2_\theta))$ (this follows from \cite{MR1645513}, although it should have been known to experts beforehand for this particular case of compact abelian group action).  The latter algebra is generated by $A$ and the operators $\int d t f_t U_t$ of unitaries $U_t$ for $t \in \T^2$ integrated with coefficient functions $f$ in $L^1(\T^2)$, $u$ and $v$ satisfying
 \begin{align*}
 \Ad_{U_t}(a) &= \sigma_t(a) &\Ad_{U_t}(u^m v^n) &= e^{2\pi i (m t_1 + n t_2)} u^m v^n\\
 [a, u^m v^n] &= 0 & u v u^* &= e^{2 \pi i \theta} v
 \end{align*}
 for any $t \in \T^2$ and $(m, n) \in \Z^2$.  This can be identified to  $\Z \ltimes_{\widehat{\sigma^{(1)}}, \alpha_\theta} \Z \ltimes_{\hat{\sigma^{(2)}}} \T^2 \ltimes_{\sigma} A$.
\end{proof}

By abuse of notation, let $\sigma^{(2)}_{\theta}$ denote the automorphism of $\T \ltimes_{\sigma^{(1)}} A$ characterized by
\[
\sigma^{(2)}_{\theta}(f)_t =\sigma^{(2)}_{\theta}(f_t) \quad (f \in L^1(\T, A), t \in \T).
\]
Then we denote by $(\widehat{\sigma^{(1)}},\sigma^{(2)}_{\theta})$ the composition of $\widehat{\sigma^{(1)}}$ and $\sigma^{(2)}_{\theta}$.

\begin{lem}\label{lem:def-cross-prod-mor-equiv}
 The algebra $A_\theta$ is strongly Morita equivalent to $\Z \ltimes_{\widehat{\sigma^{(1)}},\sigma^{(2)}_{\theta}} \T \ltimes_{\sigma^{(1)}} A$.
\end{lem}

\begin{proof}
 Let $V$ denote the unitary represented by $U^{(2)}_\theta$ in the multiplier algebra of $\Z \ltimes_{\hat{\sigma^{(2)}}} \T^2 \ltimes_{\sigma} A$.  Thus, $V$ is characterized by
 \begin{align*}
 V. f v^k &= \sigma^{(2)}_\theta(g) v^k,&  f_k v^k.V &= e^{2 \pi k i \theta}g v^k
 \end{align*}
 for $f \in L^1(\T^2, A)$, where $g$ denotes the function $f(t_1, t_2 - \theta)$.  Then $V$ is $(\widehat{\sigma^{(1)}}, \alpha_\theta)$ invariant and the crossed product of $\Z \ltimes_{\hat{\sigma^{(2)}}} \T^2 \ltimes_{\sigma} A$ by $\Z$ given by the automorphism $(\widehat{\sigma^{(1)}}, \alpha_\theta)$ and the one by $\Ad_V \circ (\widehat{\sigma^{(1)}}, \alpha_\theta)$ define isomorphic algebras.
 
 Now, one computes
 \[
\Ad_V \circ \widehat{\sigma^{(1)}} \circ \alpha_\theta (f v^k) = \Ad_V(e^{2 \pi i k \theta} h v^k) = \sigma^{(2)}_\theta(h) v^k,
 \]
 where $h(t_1, t_2) = e^{2 \pi i t_1} f(t_1, t_2)$.   hence the automorphism $\Ad_V \circ (\widehat{\sigma^{(1)}}, \alpha_\theta)$ is equal to the composition of $\widehat{\sigma^{(1)}}$ and $\sigma^{(2)}_{\theta}$.  Thus we have an isomorphism between the crossed products as
 \[
\Z \ltimes_{(\widehat{\sigma^{(1)}}, \alpha_\theta)} \T \ltimes_{\sigma^{(1)}}  \Z \ltimes_{\widehat{\sigma^{(2)}}} \T \ltimes_{\sigma^{(2)}} A \simeq \Z \ltimes_{\widehat{\sigma^{(1)}},\sigma^{(2)}_{\theta}} \T \ltimes_{\sigma^{(1)}}  \Z \ltimes_{\widehat{\sigma^{(2)}}} \T \ltimes_{\sigma^{(2)}} A.
 \]
The right hand side can be also expressed as
\[
\Z \ltimes_{\hat{\sigma^{(2)}}} \T \ltimes_{\sigma^{(2)}} \Z \ltimes_{\widehat{\sigma^{(1)}},\sigma^{(2)}_{\theta}} \T \ltimes_{\sigma^{(1)}} A.
\]
By Takesaki-Takai duality \cite{MR0365160}, this algebra is isomorphic to the tensor product of $\Z \ltimes_{\widehat{\sigma^{(1)}},\sigma^{(2)}_{\theta}} \T \ltimes_{\sigma^{(1)}} A$ with the compact operator algebra $\cptops$, which proves our assertion.
\end{proof}

\begin{rmk}\label{rmk:corresp-cl-deform-cross-prod}
The strong Morita equivalence of Lemma \ref{lem:def-cross-prod-mor-equiv} can be described in terms of a bimodule $\mathcal{E}_0$ as follows.  As a linear space, take $\mathcal{E}_0$ as the tensor product $\ell_2 \Z \otimes A$.  Define left actions of $A$, $C^*(\T)$ and $\Z$ on $\mathcal{E}_0$ by
\begin{align*}
a_{(m, n)} . \delta_k \otimes b_{(m', n')} &= \delta_{k} \otimes ab, \\
z^l . \delta_k \otimes b_{(m, n)} &= \begin{cases} \delta_k \otimes b & (l = -k + m),\\
 0& (\text{otherwise}),
 \end{cases}\\
u . \delta_k \otimes b_{(m, n)} &= e^{2 \pi i \theta n}\delta_{k+1} \otimes b,
\end{align*}
where $a_{(m, n)}, b_{(m, n)}$ denotes any element of $A$ with weight ${(m, n)}$ for $\sigma$, $z^l$ is the function $t \mapsto e^{2 \pi i k t}$ in the convolution algebra $C^* \T$ and $u$ is the generating unitary of $C^*\Z$.  These form a covariant representation of $\Z, \T$ and $A$ with respect to the action $\sigma^{(1)}$ of $\T$ on $A$ and $(\hat{\sigma}_1,\sigma^{(2)}_{\theta})$ of $\Z$ on $\T \ltimes_{\sigma^{(1)}} A$.    

Next, consider a right action of $A_\theta$ on $\mathcal{E}_0$ by
\[
\delta_k \otimes b_{(m', n')} . a^{(\theta)}_{(m, n)} = e^{2 \pi i \theta k n} \delta_{k+m} \mintensor ba.
\] 
Then $\mathcal{E}_0$ has an $A_\theta$-valued inner product $\pairing{\delta_k a_{(m, n)}}{\delta_l b_{(m', n')}} = \delta_{k-m, l-m'} (a^*b)^{(\theta)}$.  The completion $\mathcal{E}$ of $\mathcal{E}_0$ with respect to this inner product becomes a bimodule over $\Z \ltimes_{\hat{\sigma}_1,\sigma^{(2)}_{\theta}} \T \ltimes_{\sigma^{(1)}} A$ and $A_\theta$.  Moreover the operators coming from $\Z \ltimes \T \ltimes A$ are precisely the $A_\theta$-compact operators.  Hence we obtain
\[
\Z \ltimes_{\hat{\sigma}_1,\sigma^{(2)}_{\theta}} \T \ltimes_{\sigma^{(1)}} A \simeq \End^0_{A_\theta}(\mathcal{E}) \simeq \cptops \otimes A_\theta.
\]
\end{rmk}

Let $(\widehat{\sigma^{(1)}},\sigma^{(2)}_{\theta t})$ denote the action of $\R$ on $\R \ltimes_{\sigma^{(1)}} A$ given by
\[
 ((\widehat{\sigma^{(1)}},\sigma^{(2)}_{\theta t})_{t_0}f)_s = e^{2 \pi t_0 s} \sigma_{(0, \theta t_0)}(f_s).
\]

\begin{prop}\label{prop:cl-deform-continu-cprod}
 The algebra $A_\theta$ is strongly Morita equivalent to $\R \ltimes_{(\widehat{\sigma^{(1)}},\sigma^{(2)}_{\theta t})} \R \ltimes_{\sigma^{(1)}} A$.
\end{prop}

\begin{proof}
For any automorphism $\alpha$ of $A$, the mapping cone $M_\alpha A$ of $\alpha$ is defined to be the algebra of functions from $\R$ to $A$ satisfying $f(t + 1) = \alpha(f(t))$ for any $t \in \R$.  It admits a natural action $\tilde{\alpha}$, called the suspension flow, of $\R$ by $(\tilde{\alpha}_s f)_t = f(t + s)$  Regarding this flow we have the strong Morita equivalence between $\R \ltimes_{\tilde{\alpha}} M_\alpha A$ and $\R \ltimes_\alpha A$.

 The crossed product $\R \ltimes_{\sigma^{(1)}} A$ is isomorphic to the mapping cone of $\widehat{\sigma^{(1)}}$ on $\T \ltimes_{\sigma^{(1)}} A$, and the suspension flow is identified to $\widehat{\sigma^{(1)}}\colon \R \curvearrowright \R \ltimes_{\sigma^{(1)}} A$.
 
 We claim that the mapping cone of the automorphism$(\widehat{\sigma^{(1)}},\sigma^{(2)}_{\theta})$ on $\T \ltimes_{\sigma^{(1)}} A$ is isomorphic to $\R \ltimes_{\sigma^{(1)}} A$.  Indeed, there is the `untwisting isomorphism' $f \mapsto \beta f, (\beta f)_t =\sigma^{(2)}_{\theta t}(f_t)$ from $M_{\widehat{\sigma^{(1)}}} \T \ltimes_{\sigma^{(1)}} A$ to $M_{(\widehat{\sigma^{(1)}},\sigma^{(2)}_{\theta}} \T \ltimes_{\sigma^{(1)}} A$ which is compatible with the suspension flow.
\end{proof}

\begin{cor}[\cite{MR1237992}]\label{cor:k-group-isom}
 The $K$-groups of $A$ and $A_\theta$ are naturally isomorphic.
\end{cor}

\begin{proof}
 By Proposition \ref{prop:cl-deform-continu-cprod}, the $K$-group of $A_\theta$ can be identified to that of $\R \ltimes_{\widehat{\sigma^{(1)}},\sigma^{(2)}_{\theta t}} \R \ltimes_{\sigma^{(1)}} A$.  By Connes-Thom isomorphism \cite{MR605351}, the latter is naturally isomorphic to the $K$-group of $\R \ltimes_{\sigma^{(1)}} A$ with parity exchange in the degree.  This does not depend on $\theta$.
\end{proof}

\subsection{Smooth $\KK$-equivalence of CL deformations} Now we turn to the smooth analogue of the above $\KK$-equivalence between the Connes-Landi deformations.  The strong Morita equivalence of Proposition \ref{prop:cl-deform-continu-cprod} restricts to an isomorphism between the rapid decay infinite matrix algebra with coefficients in $\smooth{A}_\theta$ and the ``smooth'' crossed product $\Z \ltimes \T \ltimes \smooth{A}$.

Let $\T \ltimes_{\sigma^{(1)}} \smooth{A}$ denote the algebra given by the convolution product on the linear space $C^\infty(\T; \smooth{A})$ with convolution product twisted by $\sigma^{(1)}$.  This algebra has seminorms given by the seminorms $\nu_{k, \alpha}$ for $k \in \N$ and $\alpha \in \N^2$ on $\smooth{A}$ and the derivations with respect to the variable on $\T$.  Next, let $\Z \ltimes_{\hat\sigma^{(1)}\sigma^{(2)}(\theta)} \T \ltimes_{\sigma^{(1)}} \smooth{A}$ denote the algebra of sequences $(f_n)_{n \in \Z}$ in $ \T \ltimes_{\sigma^{(1)}} \smooth{A}$ which are rapid decay with respect to the aforementioned seminorms, endowed with the convolution product twisted by $\hat\sigma^{(1)}\sigma^{(2)}(\theta)$.  We have the inclusions $\T \ltimes_{\sigma^{(1)}} \smooth{A} \subset \T \ltimes A$ and $\Z \ltimes_{\hat\sigma^{(1)}\sigma^{(2)}(\theta)} \T \ltimes_{\sigma^{(1)}} \smooth{A} \subset \Z \ltimes \T \ltimes A$ into the corresponding $C^*$-algebras.

Let $\smooth{K}^\infty$ denote the algebra of the rapid decay matrices on $\Z$ (\cite{MR1082838}, Section 2): the elements of $\smooth{K}^\infty$ are the matrices $(a_{i, j})_{i, j \in \Z}$ indexed by $\Z$ satisfying
 \[
 \sqrt{1 + i^2 + j^2}^k a_{i, j} \rightarrow 0 \quad ( i, j \rightarrow \infty)
 \]
for $k \in \N$. The algebra $\smooth{K}^\infty$ becomes a Fr\'{e}chet algebra with respect to the family of (semi-)norms
\[
\mu_k((a_{i, j})_{i, j \in \Z} ) =  \max_{i, j \in \Z} \sqrt{1 + i^2 + j^2}^k a_{i, j},
\]
for $k \in \N$.  The projective tensor product $\smooth{K}^\infty \ptensor \smooth{A}_\theta$ is identified to the algebra of the $\smooth{A}_\theta$-valued matrices which are rapid decay with respect to the seminorms on $\smooth{A}_\theta$ (loc. cit. Corollary 2.4).

\begin{prop}\label{prop:smooth-alg-pres-z-t}
 The Fr\'{e}chet algebras $\Z \ltimes_{\widehat{\sigma^{(1)}},\sigma^{(2)}_{\theta}} \T \ltimes_{\sigma^{(1)}} \smooth{A}$ and $\smooth{K}^\infty \ptensor \smooth{A}_\theta$ are isomorphic via the correspondences
 \begin{multline}\label{eq:cross-prod-cl-deform-ker-corr}
 \Z \ltimes_{\widehat{\sigma^{(1)}},\sigma^{(2)}_{\theta}} \T \ltimes_{\sigma^{(1)}} \smooth{A} \ni u^nf(t) \mapsto k_{u^n f(t)}(l, m)_{l, m \in \Z} \\
  = \left( \int_{\T^2} e^{2 \pi i (m t - (n + m - l)s)} \sigma_{(s - t, (n-l)\theta)}(f(t)) d t d s\right)^{(\theta)}
  \end{multline}
  and
  \begin{gather*}
 \int_\T \sigma^{(1)}_{t}(a) \sum_{m \in \Z} e^{- 2 \pi i m t} u^m d t \mapsfrom a^{(\theta)} \in \smooth{A}_\theta,\\
 u^{l_0 - m_0}e^{-2 \pi i m_0 s} \mapsfrom e_{l_0, m_0} = \delta_{(l_0, m_0)}(l, m) \in \cptops^\infty.
 \end{gather*}
 Here the terms in the left hand side denote elements in $\Z \ltimes_{\widehat{\sigma^{(1)}},\sigma^{(2)}_{\theta}} \T \ltimes_{\sigma^{(1)}} \smooth{A}$ and the ones in the right hand side denote $\smooth{A}_\theta$-valued matrices in $\smooth{K}^\infty \ptensor \smooth{A}_\theta$.
\end{prop}

\begin{proof}
These formulae determine well-defined linear maps between the two algebras by the assumption on the smoothness of $\sigma$ and the description of the smooth crossed products above.

On the other hand, the multiplicativity in the assertion follows from the bimodule described in Remark \ref{rmk:corresp-cl-deform-cross-prod}.  We may take a ``basis'' $(\delta_k \otimes 1)_{k \in \Z}$ of $\mathcal{E}_0$ over the right action of $\smooth{A}_\theta$.  We indicate the case for (\ref{eq:cross-prod-cl-deform-ker-corr}) in the following.  Suppose that $f \in C^\infty(\T; \smooth{A})$ is a function of the form $t \mapsto e^{2 \pi i l t} a$ for some integer $l$ and $a \in \smooth{A}$ of weight $(m, n')$ for $\sigma$.  First we have
\[
u^n f(t) . \delta_k \otimes 1 = 
\begin{cases}
e^{2 \pi i \theta n n'} \delta_{k+n} \otimes a & (k = m - l)\\
0 & (\text{otherwise}).
\end{cases}
\]
On the other hand we have
\[
e^{2 \pi i \theta n n'} \delta_{k+n} \otimes a = (\delta_{k - m + n} \otimes 1) . (e^{2 \pi i \theta \{ (m - k - n) n' + n n'\}} a)^{(\theta)}.
\]
Hence $u^n f(t)$ represents a operator which moves the $(m-l)$-th base to the $(n - l)$-th base multiplied by $e^{2 \pi i \theta l n'} a^{(\theta)}$.  Thus we need to show that the formula (\ref{eq:cross-prod-cl-deform-ker-corr}) applied to our choice of $f$ gives $e^{2 \pi i \theta l n'} a^{(\theta)} \delta_{n - l, m - l}(l', m')$.  Now, when $l'$ and $m'$ are arbitrary integers,
\begin{equation*}\begin{split}
k_{u^n f}(l', m') &= \int_{\T^2} e^{2 \pi i (m' t - (n + m' - l')s)} \sigma_{(s - t, (n-l')\theta)}(f(t)) d t d s \\
&= \int_{\T^2} e^{2 \pi i \{m' t - (n + m' - l')s + (s - t) m + (n - l') \theta n' + l t\}} a d t d s \\
&=
\begin{cases}
e^{2 \pi i \theta (n - l') n'} a & (l + m' - m = 0, m - n - m' + l' = 0) \\
0 & (\text{otherwise}).
\end{cases}
\end{split}
\end{equation*}
 In the nontrivial case of $l + m' - m = 0$ and $m - n - m' + l' = 0$ one has $l' = n - l$ and $(n - l')n' = l n$, which is the desired relation.   The rest is proved in a similar way.
\end{proof}

There is also a continuous analogue of the above argument. Let $\smooth{K}^\infty_\R$ denote the algebra of the compact operators on $L^2(\R)$ whose integral kernel belong to the Schwartz class on $\R^2$.  The projective tensor product $\smooth{K}^\infty_\R \ptensor \smooth{A}_\theta$ is identified to the algebra of the $\smooth{A}_\theta$-valued integral kernels $I(x, y)$ which satisfy
\[
\max_{x, y} \norm{ \delta^k((x^m +y^n)(\partial_x^{m'} + \partial_y^{n'})I(x, y)) } < \infty
\]
for any $m, n, m'$ and $n'$ in $\N$.

Let $\schwfuncs^*(\R; \smooth{A}, \sigma^{(1)})$ be the convolution algebra of the $\smooth{A}$-valued Schwartz functions with respect to the action $\sigma$.  The action $(\widehat{\sigma^{(1)}},\sigma^{(2)}_{\theta t})$ restricts to a smooth action on this algebra, hence we may take the convolution algebra
\[
\schwfuncs^*(\R; \schwfuncs^*(\R; \smooth{A}, \sigma^{(1)}), (\widehat{\sigma^{(1)}},\sigma^{(2)}_{\theta t}))
\]
of the $\schwfuncs^*(\R; \smooth{A}, \sigma^{(1)})$-valued Schwartz functions.  Let $\R \ltimes_{\widehat{\sigma^{(1)}},\sigma^{(2)}_{\theta t}} \R \ltimes_{\sigma^{(1)}} \smooth{A}$ denote this algebra.  We write $f(t, \tau)$ for a typical element of this algebra, where for any fixed $\tau$, the function $t \mapsto f(t, \tau)$ is in $ \schwfuncs^*(\R; \smooth{A}, \sigma^{(1)})$.  Then one has the following analogue of Proposition \ref{prop:smooth-alg-pres-z-t}:

\begin{prop}\label{prop:smooth-alg-pres}
 The Fr\'{e}chet algebras $\R \ltimes_{\widehat{\sigma^{(1)}},\sigma^{(2)}_{\theta t}} \R \ltimes_{\sigma^{(1)}} \smooth{A}$ and $\smooth{K}^\infty_\R \ptensor \smooth{A}_\theta$ are isomorphic.  The correspondence between these two algebras are given by
  \begin{multline*}
\R \ltimes_{\widehat{\sigma^{(1)}},\sigma^{(2)}_{\theta t}} \R \ltimes_{\sigma^{(1)}} \smooth{A} \ni f_{t, \tau} \mapsto k_{f_{t, \tau}}(\lambda, \mu)= \\
\left( \int_{\R^2} e^{2 \pi i (\mu t - (\tau + \mu - \lambda)s)} \sigma_{(s - t, (\tau - \lambda)\theta)}(f_{t, \tau}) d t d s d \tau \right)^{(\theta)},
\end{multline*}
and
\[
f^k_{s, \tau} = \int_\R \sigma^{(1)}_{t}(k(\lambda, \mu)) e^{- 2 \pi i (\mu - \lambda + \tau) t + \mu s}d t d \lambda d \mu \mapsfrom k(\lambda, \mu) \in \cptops^\infty_\R \ptensor \smooth{A}_\theta.
 \]
 \end{prop}
 
\begin{cor}\label{prop:per-cyclic-isom}
 There is a natural isomorphism  $\Phi_\theta\colon \HP^*(\smooth{A}) \rightarrow \HP^*(\smooth{A}_\theta)$ between the periodic cyclic cohomology groups which is compatible with the identification of the $K$-groups of Corollary \ref{cor:k-group-isom}.
\end{cor}

\begin{proof}
On the one hand, by Elliott-Natsume-Nest \cite{MR945014} Theorem 6.2, one has natural isomorphisms
 \[
 \HP^*(\R \ltimes_{\widehat{\sigma^{(1)}},\sigma^{(2)}_{\theta t}} \R \ltimes_{\sigma^{(1)}} \smooth{A}) \simeq \HP^{*+1}(\R \ltimes_{\sigma^{(1)}} \smooth{A}) \simeq \HP^*(\smooth{A})
 \]
 compatible with the Connes-Thom isomorphism $K_*(\R \ltimes_{\widehat{\sigma^{(1)}},\sigma^{(2)}_{\theta t}} \R \ltimes_{\sigma^{(1)}} A) \simeq K_*(A)$.  On the other hand, by Proposition \ref{prop:smooth-alg-pres}, one has $ \HP^*(\R \ltimes_{\widehat{\sigma^{(1)}},\sigma^{(2)}_{\theta t}} \R \ltimes_{\sigma^{(1)}} \smooth{A}) \simeq \HP^*(\smooth{A}_\theta)$ (by the stability of $\HP^*$ for $\cptops^\infty_\R \ptensor -$, loc. cit. Theorem 4.3).  Combining these and remembering that the isomorphisms involved are transpose to the ones between $K$-groups of the $C^*$-algebraic completions, we have the assertion.
\end{proof}

\begin{rmk}
 Although we need the above continuous crossed product presentation of $\cptops^\infty_\R \ptensor \smooth{A}_\theta$ later in order to investigate the deformation of cyclic cocycles, the periodic cyclic cohomology isomorphism of Corollary \ref{prop:per-cyclic-isom} itself can be deduced from the algebra isomorphism of Proposition \ref{prop:smooth-alg-pres-z-t} and the result of \cite{MR961899}.
\end{rmk}

\section{Preservation of dimension spectrum under deformation}
We keep the notations $\smooth{A}, H, D, \sigma$ and $U$ as in the previous sections.  In general the regularity of the spectral triple is not guaranteed to be preserved under a deformation as in the case of Podle\'{s} sphere \cite{MR2117628}.  We show that in the case of Connes-Landi deformation, the regularity is well preserved.  As in the previous section consider a spectral triple $(\smooth{A}, H, D)$ endowed with an action of $\T^2$ satisfying the condition (\ref{eq:cov-rep-cl}) and we assume that $\smooth{A} = \smooth{A}^\infty$.

Let $\smooth{B}$ be the algebra generated by $\delta^k(a)$ for $a \in \smooth{A} + [D, \smooth{A}]$ and $k \in \N$.  Similarly, let $\smooth{B}_\theta$ denote the algebra generated by $\delta^k(\smooth{A}_\theta + [D, \smooth{A}_\theta])$ for $k \in \N$.

Now suppose that $D$ is $n$-summable.  Given a bounded operator $T$ on $H$, the function
\[
 \zeta_T(s) = \Tr(T \absolute{D}^{-s})
\]
is called the zeta function associated to $T$.  This is a priori holomorphic in the region $\ensemble{ z \in \C \suchthat \Re(z) > n}$. If the functions of the form $\zeta_T$ for $T \in \smooth{B}$ admit meromorphic extension to the whole plane $\C$, the dimension spectrum of $(\smooth{A}, H, D)$ is defined to be the collection of the polls of the analytic continuation of the functions of the form $\zeta_T(s)$ for $T \in \smooth{B}$ \cite{MR1334867}.

\begin{thm}
 Let $(\smooth{A}, H, D)$ be a regular spectral triple whose zeta functions $\zeta_T$ for $T \in \smooth{B}$ admit meromorphic extensions to $\C$.  Suppose that there is a smooth action $\sigma$ by $\T^2$ which is spatially implemented by a strongly continuous unitary representation $U_*$ on $H$ and satisfies $\smooth{A} = \smooth{A}^\infty$.  The  dimension spectrum of the spectral triple $(\smooth{A}_\theta, H, D)$ is equal to that of $(\smooth{A}, H, D)$.
\end{thm}

\begin{proof}
 The elements of the algebra $\smooth{B}$ satisfy the assumption of Lemma \ref{lem:part-sum-conv}.  Moreover when $a_1, \ldots, a_j$ are elements of $\smooth{A}$ and $k_1, \ldots, k_j$ are positive integers, one has
 \[
\left ( \delta^{k_1}(a_1) \cdots \delta^{k_j}(a_j) \right )^{(\theta)} = \delta^{k_1}(a_1^{(\theta)} ) \cdots \delta^{k_j}(a_j^{(\theta)}).
 \]
There is also an analogous identity involving $[D, a]$ for $a \in \smooth{A}$.  Hence the algebra $\smooth{B}_\theta$ is the collection of the operators of the form $T^{(\theta)}$ for $T \in \smooth{B}$.  Since $D$ commutes with $U_t$, the function $\zeta_T(s)$ only depends on $T_{(0, 0)}$ for any bounded operator $T$ on $H$.  Since we have the equality $T_{(0, 0)}^{(\theta)} = T_{(0, 0)}$, which leads to $\zeta_{T^{(\theta)}}(s) = \zeta_T(s)$.   This shows that   the dimension spectrum of $(\smooth{A}_\theta, H, D)$ is the same as that of $(\smooth{A}, H, D)$.
\end{proof}

\section{Deformation of invariant cyclic cocycles}

In the rest of the paper we investigate the pairing of the character of the spectral triple $(\smooth{A}_\theta, H, D)$ with $K_*(\smooth{A}_\theta)$ for varying $\theta \in \R$ under the natural isomorphisms given by Corollaries \ref{cor:k-group-isom} and \ref{prop:per-cyclic-isom}.  We start with a review of Elliott-Natsume-Nest's isomorphism $\HC^n(\smooth{A}) \rightarrow \HC^{n+1}(\R \ltimes_\sigma \smooth{A}), \phi \mapsto \#_\sigma \phi$  \cite{MR945014}.

Let $\smooth{A}$ be a Fr\'{e}chet *-algebra which is dense and close under holomorphic functional calculus inside a $C^*$-algebra $A$.  Let $\sigma$ be a smooth action of $\R$ on $\smooth{A}$.

Recall that close graded traces on graded differential algebras containing $\smooth{A}$ in the degree $0$ give cyclic cocycles on $\smooth{A}$, and conversely any cyclic $n$-cocycle can be represented as a closed graded trace on the universal differential graded algebra $\Omega(\smooth{A}) = \oplus \Omega(\smooth{A})_n$ (where $\Omega(\smooth{A})_0 = \smooth{A}$ and $\Omega(\smooth{A})_n = \smooth{A}^{\ptensor n} \oplus \smooth{A}^{\ptensor n + 1}$ for $n > 0$).

Suppose that a cyclic $n$-cocycle $\phi$ on $\smooth{A}$ is given by a closed graded trace of degree $n$, which is also denoted by $\phi$ by abuse of notation, on a graded differential algebra $\Omega$ containing $\smooth{A}$ in degree $0$.  Let $\schwfuncs^*\R$ be the convolution algebra of the Schwartz class functions on $\R$, $E$ the direct sum $\Omega(\schwfuncs^*\R)_0 \oplus \Omega(\schwfuncs^*\R)_1$.  We construct a differential graded algebra structure on $\Omega \ptensor E$ containing $\R \ltimes_\sigma \smooth{A}$ in the degree $0$ as follows.

The space $\Omega \ptensor E_0$ is naturally identified to $\R \ltimes \Omega = \schwfuncs^*(\R; \Omega)$.  There is a derivation $d\colon \Omega \ptensor E_0 \rightarrow \Omega \ptensor E$ given by
\[
d(\omega \otimes f) = (d\omega)\otimes f + (-1)^{\deg \omega} \omega \otimes d f
\]
for any homogeneous element $\omega \in \Omega$. On $\Omega \ptensor E_1$, the differential is defined by $d(\omega \otimes f d g) = (d\omega)\otimes f d g$.
The $1$-forms in $\Omega(\schwfuncs^*\R)_1$ act on $\Omega \ptensor E$ by
\[
d f (\omega \otimes g) = d(f(\omega \otimes g)) - f d (\omega \otimes g).
\]
These, together with the original product structure of $\Omega$, defines a structure of a graded differential algebra on $\Omega \ptensor E$.

Then the closed graded trace $\#_\sigma \phi$ on $\Omega_n \ptensor E_1 = \Omega \ptensor \schwfuncs^*(\R) \ptensor \schwfuncs^*(\R)$ corresponding to the cyclic $(n+1)$-cocycle $\#_\sigma \phi$ over $\R \ltimes_\sigma \smooth{A}$ is defined as follows:
\begin{equation}\label{eq:enn-closed-gr-trace-def}
\#_\sigma \phi (f) = 2 \pi i \int_{-\infty}^\infty \int_0^t \phi(\sigma_s f(-t, t)) d s d t.
\end{equation}
When $\phi$ is invariant under $\sigma$, (\ref{eq:enn-closed-gr-trace-def}) reduces to
\begin{equation}\label{eq:invar-phi-sharp}
2 \pi i \int_{-\infty}^\infty t \phi(f(-t, t)) d t.
\end{equation}

This correspondence of $\phi$ to $\#_\sigma \phi$ is compatible with the Connes-Thom isomorphism $\Phi_\sigma\colon K_*A \rightarrow K_{*+1} \R \ltimes_\sigma A$ (\cite{MR945014}, Theorem 6.2). 

Let $\delta$ denote the generator
\[
\delta(a) = \lim_{t \rightarrow 0} \frac{d}{d t} \sigma_t(a)
\]
of $\sigma$.  When $\phi$ is a $\sigma$-invariant cyclic $n$-cocycle
\[
 \phi(f^0, \ldots, f^n) =  \phi(\sigma_t(f^0), \ldots, \sigma_t(f^n)) \quad (\forall t \in \T^2)
\]
 over $\smooth{A}$, we obtain a new $n+1$-cocycle $i_\delta \phi$ (\cite{MR1303779} Chapter 3, Section 6.$\beta$) by
\[
i_\delta \phi(a_0 d a_1 \cdots d a_{n+1}) = \sum_{j=1}^{n+1} (-1)^{j} \phi(a_0 d a_1 \cdots \delta(a_j) \cdots d a_{n+1}).
\]
On the other hand, noting that the convolution algebra of the $\smooth{A}$-valued rapid decay functions $\schwfuncs^*(\R; \smooth{A})$ is identified to $\R \ltimes_\sigma \smooth{A}$, one obtains the dual cocycle $\hat{\phi}$ over  $\R \ltimes_\sigma A$ by
\[
\hat{\phi}(f^0, \ldots, f^n) = \int_{\sum_{j=0}^n t_j=0} \phi(f^0_{t_0}, \sigma_{t_0}(f^1_{t_1}), \ldots, \sigma_{\sum_{j < n} t_j}(f^0_{t_n}))
\]
for $f^j \in \schwfuncs^*(\R; A)$. We also have the dual action $\hat{\sigma}\colon \R \curvearrowright \R \ltimes_\sigma \smooth{A}$ and the isomorphism $\R \ltimes_{\hat{\sigma}} \R \ltimes_\sigma \smooth{A} \simeq \cptops^\infty_\R \ptensor \smooth{A}$, which implies $\HC^k(\R \ltimes_{\hat{\sigma}} \R \ltimes_\sigma \smooth{A}) \simeq \HC^k(\smooth{A})$.  Regarding these constructions one has the following generalization of \cite{MR945014}, Proposition 3.11.

\begin{prop}\label{prop:dual-cocycle}
Let $\phi$ be a $\sigma$-invariant $n$-cocycle over $\smooth{A}$.  The class of the cyclic $(n+1)$-cocycle $i_\delta \phi$ in $\HC^{n+1}(\smooth{A})$ agrees with that of $\#_{\hat\sigma} \hat\phi$.
\end{prop}

\begin{proof}
Let $b_0, \ldots, b_{n+1}$ be elements of $\R \ltimes_\sigma \smooth{A}$ and $f_0, \ldots, f_{n+1}$ be functions in $\schwfuncs^*{\R}$.  For each $0 \le j \le n+1$, $b_j \otimes f_j$ represents the element $t \mapsto b_j \hat{f_j}(t) \in \schwfuncs(\R; \smooth{A})$.  We consider the element
\begin{equation}\label{eq:omega-form-def}
\omega = (b_0 \otimes f_0) d(b_1 \otimes f_1) \cdots d(b_{n+1} \otimes f_{n+1})
\end{equation}
of $\Omega \ptensor E$.  Each term in the above can be expanded as $d(b_j \otimes f_j) = d b_j \otimes f_j + b_j \otimes d f_j$.  Since the terms containing more than one $d f_j$'s vanish and the only terms contributing to the pairing of $\omega$ with $\#_{\hat{\sigma}} \hat{\phi}$ are the elements in $\Omega_n \ptensor E_1$.  Hence $\#_{\hat{\sigma}} \hat{\phi}(\omega)$ can be expressed as
\begin{equation}\label{eq:sharp-phi-sum}
\sum_{j=1}^{n+1} 2 \pi i \int_{s_{n+1}=0} t_j \hat{\phi}(\eta_j(t_0, \ldots, t_{n+1})) \xi_j(t_0, \ldots, t_{n+1}) d t_0 \cdots d t_n,
\end{equation}
where 
\begin{equation*}\begin{split}
\eta_j(t_0, \ldots, t_{n+1}) &= b_0 \hat{\sigma}_{s_0}(d b_1) \cdots \hat{\sigma}_{s_{j-1}}(b_j) \cdots \hat{\sigma}_{s_n}(d b_{n+1}),\\
\xi_j(t_0, \ldots, t_{n+1}) &= f_0(t_0) \cdots \check{f_j} \cdots f_{n+1}(t_{n+1}) d f_j(t_j),\\
s_k &= t_0 + \cdots + t_k.
\end{split}\end{equation*}

The derivation $\tilde{\delta} f(t) = t f(t)$ on $\schwfuncs^*(\R; \R \ltimes \smooth{A})$ is the generator of the double dual action $\hat{\hat{\sigma}}\colon \R \curvearrowright \R \ltimes_{\hat{\sigma}} \R \ltimes_\sigma \smooth{A}$.  The $j$-th term of (\ref{eq:sharp-phi-sum}) is equal to
\[
\hat{\hat{\phi}}(b_0 \otimes f_0) d(b_1 \otimes f_1) \cdots \tilde{\delta}(b_j \otimes f_j) \cdots d(b_{n+1} \otimes f_{n+1}).
\]
Collecting these terms we obtain $\#_\sigma \hat{\phi}(\omega) = i_{\tilde{\delta}} \hat{\hat{\phi}}(\omega) =  i_{\tilde{\delta}}(\phi \otimes \Tr)(\omega)$.

There is a one-parameter unitary $u_t$ such that the double dual action $\hat{\hat{\sigma}}$ is conjugated to $\Id_\cptops \otimes \sigma$ on $\cptops^\infty_\R \ptensor \smooth{A} \simeq \R \ltimes_{\hat{\sigma}} \R \ltimes_\sigma \smooth{A}$ by the formula $\hat{\hat{\sigma}}_t = \Ad_{u_t} \circ (\Id_{\cptops} \sigma_t)$.  By Connes's $2 \times 2$-matrix trick, we obtain an action $\Phi(\sigma, u_*)$ of $\R$ on $M_2( \R \ltimes_{\hat{\sigma}} \R \ltimes_\sigma \smooth{A})$ by
\[
\Phi(\sigma, u_*)_t \left( \left[\begin{array}{cc}x_{11} & x_{12} \\x_{21} & x_{22}\end{array}\right] \right)
=  \left[\begin{array}{cc}\sigma_t(x_{11}) & \sigma_t(x_{12})u_t^* \\u_t \sigma_t(x_{21}) & \Ad_{u_t} \sigma_t(x_{22})\end{array}\right] .
\]
The generator of this action can be written as
\begin{equation}\label{eq:mat-deriv}
\delta_{\sigma, u_*}\left(\left[\begin{array}{cc}x_{11} & x_{12} \\x_{21} & x_{22}\end{array}\right] \right)
= \left[\begin{array}{cc} \delta(x_{11}) & \delta(x_{12}) - x_{12} h \\ h x_{21} + \delta(x_{21}) & \tilde{\delta}(x_{22})\end{array}\right],
\end{equation}
where $h$ is the generator
\[
h = \lim_{t \rightarrow 0} \frac{u_t - 1}{t}
\]
of $u_t$ which is in the multiplier algebra of $ \R \ltimes_{\hat{\sigma}} \R \ltimes_\sigma A$, so that we have $\tilde{\delta}(x) = \delta(x) + [h, x]$.

Now, the derivation $\delta_{\sigma, u_*}$ of (\ref{eq:mat-deriv}) determines a cyclic $n+1$-cocycle $\psi = i_{\delta_{\sigma, u_*}} (\phi \otimes \Tr)$ on $M_2( \R \ltimes_{\hat{\sigma}} \R \ltimes_\sigma \smooth{A})$.  There are two embeddings of $ \R \ltimes_{\hat{\sigma}} \R \ltimes_\sigma \smooth{A}$ into  $M_2( \R \ltimes_{\hat{\sigma}} \R \ltimes_\sigma \smooth{A})$, given by
\begin{align*}
\Psi_1 (x) &= \left[\begin{array}{cc} x & 0 \\ 0 & 0 \end{array}\right], & 
\Psi_2 (x) &= \left[\begin{array}{cc} 0 & 0 \\ 0 & x \end{array}\right].
\end{align*}
The pullback $\HC^*(M_2( \R \ltimes_{\hat{\sigma}} \R \ltimes_\sigma \smooth{A})) \rightarrow \HC^*( \R \ltimes_{\hat{\sigma}} \R \ltimes_\sigma \smooth{A})$ by these two homomorphisms in cyclic cohomology become the same map.  But the pullback of $\psi$ by $\Psi_1$ is $i_\delta \phi \otimes \Tr$ while the one by $\Psi_2$ is $i_{\tilde{\delta}} (\phi \otimes \Tr)$, hence these two cocycles determine the same class in the cyclic cohomology group, which implies $i_\delta \phi = \#_\sigma \hat{\phi}$ in $\HC^*(\R \ltimes_{\hat{\sigma}} \R \ltimes_\sigma \smooth{A})$.
\end{proof}

Now we consider an even regular spectral triple $(\smooth{A}, H, D)$ with an action of $\T^2$ satisfying $\smooth{A}^\infty = \smooth{A}$ as in the previous sections.  Note that we have the estimates
\[
\nu_{l, \alpha}(\frac{d^k}{d t^k} \sigma^{(i)}_t(a)) = \nu_{l, \alpha'}(a)
\]
for $k \in \N$ and $\alpha \in \N^2$, where $\alpha' = (\alpha_1 + k, \alpha_2)$ or $\alpha' = (\alpha_1, \alpha_2 + k)$ corresponding to the cases $i = 1, 2$.  Hence the actions $\sigma^{(i)}$ are smooth action on $\smooth{A}^\infty = \smooth{A}$ in the sense above.

Let $\phi$ be a cyclic $n$-cocycle on $\smooth{A}$ which is invariant under $\sigma$.  Then we obtain the dual cocycle $\hat{\phi}$ over the crossed product $\T \ltimes_{\sigma^{(1)}} \smooth{A}$.  Then again, $\hat{\phi}$ is invariant under the action $(\widehat{\sigma^{(1)}},\sigma^{(2)}_{\theta})$.  Then we obtain a cocycle $\hat{\hat{\phi}}$ on $\Z \ltimes_{\widehat{\sigma^{(1)}},\sigma^{(2)}_{\theta}} \T_{\sigma^{(1)}} \ltimes \smooth{A}$, and with a similar process another cyclic cocycle on $\R \ltimes_{\widehat{\sigma^{(1)}},\sigma^{(2)}_{\theta t}} \R \ltimes_{\sigma^{(1)}} \smooth{A}$, which we still denote by $\hat{\hat{\phi}}$.  The cocycle $\hat{\hat \phi}$ induces a one $\phi^{(\theta)}$ on $\smooth{A}_\theta$ via the embedding given into a corner of
\[
\cptops^\infty \ptensor \smooth{A}_\theta \simeq \Z \ltimes_{\widehat{\sigma^{(1)}},\sigma^{(2)}_{\theta}} \T_{\sigma^{(1)}} \ltimes \smooth{A}
\]
by Proposition \ref{prop:smooth-alg-pres-z-t}.  We record the formula for $\phi^{(\theta)}$ for $a \in \smooth{A}$:
 \begin{multline}\label{eq:inv-cocycle-twisted-over-deform}
  \phi^{(\theta)}(a_0^{(\theta)}, \ldots, a_n^{(\theta)}) \\
   = \sum_{m_0+\cdots + m_n = 0, l_0 + \cdots + l_n =0} \phi((a_0)_{m_0, n_0}, b_{(1, m_1, l_1)}, \ldots, b_{(n, m_n, l_n)}),
\end{multline}
where $b_{(k, m_k, l_k)} = e^{2 \pi i \theta (\sum_{j < k} m_j) l_k} (a_k)_{m_k, l_k}$.  In the particular case of $n = 0$, $\phi$ is given by a trace on $\smooth{A}$ and $\phi^{(\theta)}$ is given by the corresponding trace $\phi^{(\theta)}(a^{(\theta)}) = \phi(a_{(0, 0)}) = \phi(a)$.

\begin{thm}\label{prop:double-dual-cocycle-theta-deform}
 Let $\phi$ be a $\sigma$-invariant cyclic $n$-cocycle on $\smooth{A}$.  Then the cyclic $n$-cocycle $\phi^{(\theta)}$ on $\smooth{A}_\theta$ corresponds to the nonhomogeneous cyclic cocycle
 \begin{equation}\label{eq:deformed-cocycle-corr-orig-alg}
 \phi + \theta i_{\delta^{(1)}} i_{\delta^{(2)} } \phi
\end{equation}
 on $\smooth{A}$ under the natural isomorphism of Corollary \ref{prop:per-cyclic-isom}.
\end{thm}

\begin{proof}
It is enough to check that the cyclic $n$-cocycle $\hat{\hat \phi}$ on $\R \ltimes_{\widehat{\sigma^{(1)}},\sigma^{(2)}_{\theta t}} \R \ltimes_{\sigma^{(1)}} \smooth{A}$ corresponds to the one of (\ref{eq:deformed-cocycle-corr-orig-alg}) on $\smooth{A}$ via the natural isomorphism in $\HP^*$.  The action $(\hat{\sigma}_1,\sigma^{(2)}_{\theta t})$ of $\R$ on $\R \ltimes_{\sigma^{(1)}} A$ is generated by the derivation
 \[
 \delta^{(\theta)}(f)_s = s f_s + \theta \frac{d}{d t_2} \sigma_{(0, t_2)}(f_s).
 \]
Let $\delta^{\hat{(1)}}$ denote the generator of the dual action $\widehat{\sigma^{(1)}}$ and $\delta^{(2)}$ that of the action $\sigma^{(2)}$.  The equation above shows that $ \delta^{(\theta)} =\widehat{\sigma^{(1)}} + \theta \delta^{(2)}$.  Applying Proposition \ref{prop:dual-cocycle}, the cocycle $\hat{\hat{\phi}}$ corresponds to the cyclic cocycle
\[
i_{\delta^{(\theta)}} \hat{\phi} = i_{\delta^{\hat{(1)}}} \hat{\phi} + \theta i_{\delta^{(2)}} \hat{\phi}
\]
on $\R \ltimes_{\sigma^{(1)}} \smooth{A}$.

The first term $i_{\delta^{\hat{(1)}}} \hat{\phi}$ in the right hand side of the above is the cocycle which corresponds to $\phi$ under the isomorphism $\HP^*(\smooth{A}) \simeq \HP^{*+1}(\R \ltimes_{\sigma^{(1)}} \smooth{A})$.  On the other hand, one has $i_{\delta^{(2)}} \hat{\phi} = \widehat{ i_{\delta^{(2)}} \phi}$ for the second term.  This $n+1$-cocycle on $\R \ltimes_{\sigma^{(1)}} \smooth{A}$ corresponds to the $n+2$-cocycle $i_{\delta^{(1)}} i_{\delta^{(2)}} \phi$ on $\smooth{A}$ again by Proposition \ref{prop:dual-cocycle}.  Combining these two, one obtains the assertion of the Proposition.
\end{proof}

\begin{example}
 Let $M$ be a compact smooth manifold, endowed with a smooth action $\sigma$ of $\T^2$.  Then $M$ admits a Riemannian metric which is invariant under $\sigma$.  Then the algebra of the smooth functions on $M$, the de Rham complex $\Omega^*(M)$ graded by the degree of forms and the operator $d + d^*$ densely defined on $H = L^2(\Omega^*(M))$ and the induced representation of $\T^2$ on $H$ satisfies the assumption of this paper.
 
 The volume form $d v$ on $M$ is invariant under $\sigma$ and it defines an invariant trace $\tau\colon f \mapsto \int_M f d v$ on $C(M)$.  Let $X_i$ denote the vector fields on $M$ generating $\sigma_i$ for $i = 1, 2$.  By the preceding proposition, the map $K_0(C M_\theta) \rightarrow \C$ induced by the trace $\tau^{(\theta)}$ on $C M_\theta$ corresponds to the map $K^0(M) \rightarrow \C$ induced by the current
\[
f + g^0 d g^1 \wedge d g^2 \mapsto \int_M f + \theta ( g^0 \pairing{d g^1}{X_1} \pairing{d g^2}{X_2} -  \pairing{d g^1}{X_2}  \pairing{d g^2}{X_1} ) d v.
\]
When $E$ is a vector bundle over $M$, the above map on the class of $E$ in $K^0(M)$ gives the number
\[
 \int_M \ch(E)\wedge( d v + \theta i_{X_1} i_{X_2} dv ) = \vol(M) \rk(E) + \theta \int_M c_1(E) \wedge i_{X_1} i_{X_2} dv.
\]
 \end{example}

\subsection{Chern-Connes character of deformed triple}\label{subsec:cc-char-deform-triple}

The algebra $\R \ltimes_{\sigma^{(1)}} \smooth{A}$ acts on the Hilbert space $L^2(\R) \htensor H$ in the following way: a function $f$ in $\schwfuncs^*(\R, \smooth{A})$ transforms vectors in $L^2(\R; H) \simeq L^2(\R) \htensor H$ by
\[
 (\pi(f) \xi)_t = \int_\R f_s U_{(s, 0)} \xi_{t-s} d s.
\]

Next we define an action of the algebra $\R \ltimes_{\widehat{\sigma^{(1)}},\sigma^{(2)}_{\theta t}} \R \ltimes_{\sigma^{(1)}} \smooth{A}$ on $L^2(\R; H)$ by
\[
(\hat{\pi}(f) \xi)_t = \int_\R e^{2\pi i s t} (\pi({f^s})U_{(0, \theta s)}\xi)_t d s
\]
for $f \in \schwfuncs^*(\R; \R \ltimes_{\sigma^{(1)}} \smooth{A})$.
  
\begin{lem}
 The correspondence $f \mapsto \hat{\pi}(f)$ defines a representation of $\R \ltimes_{\widehat{\sigma^{(1)}},\sigma^{(2)}_{\theta t}} \R \ltimes_{\sigma^{(1)}} \smooth{A}$ on $L^2(\R) \htensor H$.
\end{lem}

    Similarly, we have a representation of $\Z \ltimes_{\widehat{\sigma^{(1)}},\sigma^{(2)}_{\theta}} \T \ltimes_{\sigma^{(1)}} \smooth{A}$ on $\ell_2 \Z \htensor H$ by
  \[
  z^l a \delta_k \otimes \xi = 
  \begin{cases}
  \delta_k \otimes a \xi & (a\xi \in H_{(m, n)}, l + k - m = 0)\\
  0 & (\text{otherwise})
  \end{cases}
  \]
  and $w \delta_k \otimes \xi = e^{2 \pi i n \theta} \delta_{k + 1} \otimes \xi$ for any $\xi \in H_{(m, n)}$.

In the following we make a more detailed analysis of the strong Morita equivalence given by Proposition \ref{prop:cl-deform-continu-cprod} in relation to the unbounded selfadjoint operator $D$.

\begin{prop}\label{prop:unitary-conj-crossed-tensor-cl}
 There is a unitary operator $U_0$ on $\ell_2(\Z) \htensor H$ satisfying $U_0^* (1 \otimes D) U_0 = 1 \otimes D$ and
\[
U_0^* (\Z \ltimes_{\widehat{\sigma^{(1)}},\sigma^{(2)}_{\theta}} \T \ltimes_{\sigma^{(1)}} A) U_0 = \cptops(\ell_2(\Z)) \otimes \smooth{A}_\theta.
\]
\end{prop}

\begin{proof}
Let $U_0$ denote the unitary operator on $\ell_2 \Z \htensor H$ given by $U_0 \delta_k \otimes \xi = e^{2\pi i \theta n k} \delta_{k+m} \otimes \xi$ for any $k \in \Z$ and $\xi \in H_{(m, n)}$.

When $a \in A_{(p, q)}$ and $\xi \in H_{m, n}$, we have
\begin{equation*}\begin{split}
U_0^* z^l a U_0 \delta_k \otimes \xi &= U_0 z^l a e^{2 \pi i \theta k n} \delta_{k+m} \otimes \xi \\
&=
\begin{cases}
e^{2 \pi i \theta q(p-k) + p n} \delta_{k - p} \otimes a\xi & (l + k - p = 0)\\
0 & (\text{otherwise}).
\end{cases}
\end{split}
\end{equation*}
This is the effect of $e^{2 \pi i \theta q l} e_{-l, -(l+p)} \otimes a^{(\theta)} \in \cptops(\ell_2 \Z) \mintensor A_\theta$ on $\delta_k \otimes \xi$, where $e_{m, n}$ denotes the matrix element $\delta_k \mapsto \delta_{n, k} \delta_m$ on $\ell_2 \Z$ for any $(m, n) \in \Z^2$.

Similarly, one has $U_0^* w U_0 = v \otimes 1$ where $v$ is the unitary operator $\delta_k \mapsto \delta_{k+1}$ on $\ell_2 \Z$.  Hence we have $U_0^* \Z \ltimes_{\widehat{\sigma^{(1)}},\sigma^{(2)}_{\theta}} \T \ltimes_{\sigma^{(1)}} A U_0 \subset \cptops(\ell_2(\Z)) \mintensor A_\theta$.  One can similarly show that $U_0 (\cptops(\ell_2 \Z) \otimes \smooth{A}_\theta) U_0^* \subset \Z \ltimes_{\widehat{\sigma^{(1)}},\sigma^{(2)}_{\theta}} \T \ltimes_{\sigma^{(1)}} A$.
\end{proof}

As a consequence of Proposition \ref{prop:unitary-conj-crossed-tensor-cl} the image under $\hat{\pi}$ of the smooth subalgebra $\Z \ltimes_{\widehat{\sigma^{(1)}},\sigma^{(2)}_{\theta}} \T \ltimes_{\sigma^{(1)}} \smooth{A}$ of $\Z \ltimes_{\widehat{\sigma^{(1)}},\sigma^{(2)}_{\theta}} \T \ltimes_{\sigma^{(1)}} A$ have bounded commutators with $1 \otimes D$.  In particular,
\begin{equation}\label{eq:dfn-spec-triple-twist-crossed-prod}
(\Z \ltimes_{\widehat{\sigma^{(1)}},\sigma^{(2)}_{\theta}} \T \ltimes_{\sigma^{(1)}} \smooth{A}, \ell_2 \Z \htensor H, 1 \otimes D)
\end{equation}
is a spectral triple over $\Z \ltimes_{\widehat{\sigma^{(1)}},\sigma^{(2)}_{\theta}} \T \ltimes_{\sigma^{(1)}} \smooth{A}$.

\begin{cor}
 The character of the spectral triple (\ref{eq:dfn-spec-triple-twist-crossed-prod}) corresponds to the map induced by the character of $D$ over $\smooth{A}_\theta$ via the strong Morita equivalence of Lemma \ref{lem:def-cross-prod-mor-equiv}.
\end{cor}

\begin{proof}
 The isomorphism of $\R \ltimes_{\widehat{\sigma^{(1)}},\sigma^{(2)}_{\theta t}} \R \ltimes_{\sigma^{(1)}} A$ onto $\cptops^\infty_\R \ptensor \smooth{A}_\theta$ given by $\Ad_U$ is equal to the strong Morita equivalence of Lemma \ref{lem:def-cross-prod-mor-equiv}.  By the $U$-invariance of $1 \otimes D$, its character over $\R \ltimes_{\widehat{\sigma^{(1)}},\sigma^{(2)}_{\theta t}} \R \ltimes_{\sigma^{(1)}} \smooth{A}$ is identified to that over $\cptops^\infty_\R \ptensor \smooth{A}_\theta$.  The latter induces the same class as the character of $D$ in the cyclic cohomology group via the isomorphism $\HC^*(\cptops^\infty_\R \ptensor \smooth{A}_\theta) \simeq \HC^*(\smooth{A}_\theta)$.
\end{proof}

Suppose that $(\smooth{A}, H, D)$ is $n$-summable ($0 < n \in 2 \N$).  Then the character
\[
\ch_D(f_0, \ldots, f_n) = \Tr_s(\gamma f_0 [F, f_1] \cdots [F, f_n]) \quad (F = D \absolute{D}^{-1})
\]
of this triple for is a cyclic $n$-cocycle which is invariant $\sigma$.

\begin{lem}
 The character of the spectral triple (\ref{eq:dfn-spec-triple-twist-crossed-prod}) over $\Z \ltimes_{\widehat{\sigma^{(1)}},\sigma^{(2)}_{\theta}} \T \ltimes_{\sigma^{(1)}} A$ is equal to $\hat{\hat{\ch_D}}$.
\end{lem}

\begin{cor}\label{cor:char-sp-triple-continu-crossed-prod}
 Under the strong Morita equivalence of Proposition \ref{prop:cl-deform-continu-cprod}, the character of $D$ over $\smooth{A}_\theta$ corresponds to $\hat{\hat{\ch_D}}$ over $\R \ltimes_{\widehat{\sigma^{(1)}},\sigma^{(2)}_{\theta t}} \R \ltimes_{\sigma^{(1)}} \smooth{A}$.
\end{cor}

\begin{rmk}
 As a consequence of Corollary \ref{cor:char-sp-triple-continu-crossed-prod}, or directly from the formula (\ref{eq:inv-cocycle-twisted-over-deform}), the cocycle $\ch_{(\smooth{A}, H, D)}^{(\theta)}$ over $\smooth{A}_\theta$ agrees with $\ch_{(\smooth{A}_\theta, H, D)}$.
\end{rmk}
 
\begin{thm}\label{thm:cl-deform-pairing-invar}
 Let $(\smooth{A}, H, D)$ be an even regular spectral triple endowed with a smooth action of $\T^2$ satisfying $\smooth{A} = \smooth{A}^\infty$. Then the Chern-Connes characters of $(\smooth{A}_\theta, H, D)$ and $(\smooth{A}, H, D)$ induce the same maps on $K_0(A_\theta)$ via the isomorphism $\Lambda_\theta$ of Corollary \ref{cor:k-group-isom}.
\end{thm}
 
\begin{proof}
Put $\phi = \ch_D$.  By Corollary \ref{cor:char-sp-triple-continu-crossed-prod}, the character of $D$ over $\smooth{A}_\theta$ corresponds to the cocycle $\hat{\hat{\phi}}$ on $\R \ltimes_{\widehat{\sigma^{(1)}},\sigma^{(2)}_{\theta t}} \R \ltimes_{\sigma^{(1)}} \smooth{A}$.  As in the proof of Proposition \ref{prop:double-dual-cocycle-theta-deform}, this cocycle corresponds to the cocycle $i_{\delta^{\hat{(1)}}} \hat\phi + \theta i_{\delta^{(2)}} \hat \phi$ on $\R \ltimes \smooth{A}$.

Given an action $\alpha$ on a $C^*$-algebra $B$, let $\Phi_\alpha$ denote the Connes-Thom isomorphism \cite{MR605351} $K_*B \rightarrow K_{*+1}(\R \ltimes_\alpha B)$.  Let $y$ be any element of $K_1(\R \ltimes A)$.  Then one has
\[
\pairing{\hat{\hat \phi}}{\Phi_{\widehat{\sigma^{(1)}}\sigma^{(2)}_{\theta t}}^{-1}(y)} = \pairing{i_{\delta^{\hat{(1)}}} \hat\phi + \theta i_{\delta^{(2)}} \hat \phi}{y} = \pairing{i_{\delta^{\hat{(1)}}} \hat\phi}{y}  + \theta \pairing{ i_{\delta^{(2)}} \hat \phi}{y}.
\]
Since $\hat{\hat \phi}$ is equal to a character of a spectral triple, its pairing with any element in $K_0(\R \ltimes_{\widehat{\sigma^{(1)}},\sigma^{(2)}_{\theta t}} \R \ltimes_{\sigma^{(1)}} A)$ must be an integer.  Hence $ \pairing{i_{\delta^{\hat{(1)}}} \hat\phi}{y}  + \theta \pairing{ i_{\delta^{(2)}} \hat \phi}{y}$ must stay inside $\Z$ regardless of the value of $\theta$, which implies $ \pairing{ i_{\delta^{(2)}} \hat \phi}{y} = 0$.

Let $\Psi\colon K_0(A_\theta) \rightarrow K_0(\R \ltimes_{\widehat{\sigma^{(1)}},\sigma^{(2)}_{\theta t}} \R \ltimes_{\sigma^{(1)}} A)$ be the isomorphism given by Proposition \ref{prop:cl-deform-continu-cprod}.  For any element $x \in K_0(A_\theta)$, one has
\[
\pairing{\ch_D}{x} = \pairing{\hat{\hat \phi}}{\Psi(x)} = \pairing{i_{\delta^{\hat{(1)}}} \hat\phi}{\Phi_{\widehat{\sigma^{(1)}}\sigma^{(2)}_{\theta t}}\Psi(x)} = \pairing{\ch_D}{\Phi_{\sigma^{(1)}} \Phi_{\widehat{\sigma^{(1)}}\sigma^{(2)}_{\theta t}}\Psi(x)}.
\]
Since $x \mapsto \Phi_{\sigma^{(1)}} \Phi_{\widehat{\sigma^{(1)}}\sigma^{(2)}_{\theta t}}\Psi(x)$ is the natural isomorphism $\Lambda_\theta\colon K_0(A_\theta) \rightarrow K_0(A)$, one obtains the assertion of the theorem.
\end{proof}

\begin{rmk}
In the proof of Theorem \ref{thm:cl-deform-pairing-invar} we saw that the cyclic $(n+2)$-cocycle $i_{\delta^{(1)}} i_{\delta^{(2)}} \ch_D$ on $\smooth{A}$ pairs trivially with $K_0(A)$.  It is very likely that this cocycle gives the trivial class of $\HP^0(\smooth{A})$.  Note that this phenomenon is specific to the cases where one has a unitary representation $U_t$ on $H$ satisfying (\ref{eq:cov-rep-cl}).  Otherwise, $i_{\delta^{(1)}} i_{\delta^{(2)}} \ch_D$ can pair nontrivially with $K_0(A)$.  For example, the trace $\tau$ on $\T^2$ given by the Haar integral is a character of a $2$-summable even spectral triple over $C^\infty(\T^2)$ and one has $\pairing{i_{\delta^{(1)}} i_{\delta^{(2)}} \tau}{K^0 \T^2} = \Z$.
\end{rmk}

\begin{rmk}
 After a major portion of the results in this paper was obtained the author has learned from N. Higson that the invariance of the index pairing as stated in Theorem \ref{thm:cl-deform-pairing-invar} can be explained in a purely $C^*$-algebraic framework as follows: let $\R$ act on the $C^*$-algebra $\R \ltimes_{\sigma^{(1)}} A \mintensor C[0, 1]$ by $(\widehat{\sigma^{(1)}},\sigma^{(2)}_{\theta t})$ at the fiber of $\theta \in [0, 1]$.  Then the resulting crossed product algebra $\R \ltimes (\R \ltimes_{\sigma^{(1)}} A \mintensor C[0, 1])$ act on the Hilbert space $L^2(\R) \htensor H \htensor L^2([0, 1])$ determined by the action of $\R \ltimes_{\widehat{\sigma^{(1)}},\sigma^{(2)}_{\theta t}} \R \ltimes_{\sigma^{(1)}} A $ as described at the beginning of Section \ref{subsec:cc-char-deform-triple} and the pointwise multiplication representation of $C[0, 1]$ on $L^2[0, 1]$.  Then the $\R \ltimes (\R \ltimes_{\sigma^{(1)}} A \mintensor C[0, 1])$-$C[0, 1]$-module $L^2(\R) \htensor H \htensor L^2([0, 1])$ together with the unbounded $C[0, 1]$-endomorphism $1_{L^2(\R)} \otimes D \otimes 1_{L^2([0, 1])}$ define an element $\alpha$ of $\KK(\R \ltimes (\R \ltimes_{\sigma^{(1)}} A \mintensor C[0, 1], C[0, 1])$.  The evaluation at $\theta \in [0, 1]$ gives us $\KK$-equivalences
 \begin{equation*}
 \begin{split}
\ev_\theta&\colon \R \ltimes (\R \ltimes_{\sigma^{(1)}} A \mintensor C[0, 1] ) \simeq_{\KK} \R \ltimes_{\widehat{\sigma^{(1)}},\sigma^{(2)}_{\theta t}} \R \ltimes_{\sigma^{(1)}} A , \\
\ev_\theta&\colon  C[0, 1] \simeq_{\KK} \C.
\end{split}
\end{equation*}
 These $\KK$-equivalences intertwine $\alpha$ and the element of $K^0(A_\theta)$ given by $D$, while the composition $\ev_0 \circ \ev_\theta^{-1}$ gives the natural isomorphism between $K_0(A_\theta)$ and $K_0(A)$.
\end{rmk}

\paragraph{Acknowledgement} The author would like to thank Y. Kawahigashi for his continuing support.  He is also grateful to N. Higson, E. Blanchard, G. Skandalis, R. Tomatsu, N. Ozawa, R. Ponge, T. Natsume who gave many insightful suggestions during the various stages of research.  He also benefitted from conversations with C. Oikonomides, T. Fukaya, S. Oguni and many others. % Lastly but not least, he is deeply indebted to the referees who took the painful job of reading his early manuscript and made many helpful comments.

\bibliography{mybibliography}{}
\bibliographystyle{math}

\end{document}